\documentclass{article}

\title{Log-concavity and Approximate Counting for Totally Unimodular Polytopes}
\author{Jonathan Leake and Maryam Mohammadi Yekta}
\date{September 2026}

\usepackage[utf8]{inputenc}
\usepackage{comment}
\usepackage{amsmath}
\usepackage{amsfonts}
\usepackage{amssymb}
\usepackage{tikz-cd}
\usepackage{color,soul}
\usepackage{array}
\usepackage{zref-savepos}
\usepackage{algorithm}
\usepackage{algpseudocode}
\usepackage{amsthm}
\usepackage{multirow, bigstrut}
\usepackage[font={small,it}]{caption}
\usepackage{tikz}
\usepackage{bm}
\usetikzlibrary{arrows}
\usepackage{mathtools}

\usepackage{hyperref}
\hypersetup{
    colorlinks=true,
    linkcolor=blue,
    filecolor=magenta,  
    urlcolor=cyan,
}
\usepackage{cleveref}
\crefname{fact}{Fact}{Facts}
\Crefname{fact}{Fact}{Facts}
\usetikzlibrary{calc}
\usepackage{comment}
\usepackage{enumerate}
\usepackage{tikz}
\usepackage{amsmath,accents}
\usetikzlibrary{shapes.geometric}

\newcommand\restr[2]{{
  \left.\kern-\nulldelimiterspace 
  #1 
  \vphantom{\big|} 
  \right|_{#2} 
  }}

\DeclarePairedDelimiter{\norm}{\lVert}{\rVert}
\newtheorem{theorem}{Theorem}[section]
\newtheorem*{theorem*}{Theorem}
\newtheorem{lemma}[theorem]{Lemma}
\newtheorem{question}[theorem]{Question}

\newtheorem{proposition}[theorem]{Proposition}
\newtheorem{corollary}[theorem]{Corollary}

\theoremstyle{definition}
\newtheorem{definition}[theorem]{Definition}
\newtheorem{example}[theorem]{Example}
\newtheorem{remark}[theorem]{Remark}
\newtheorem{fact}[theorem]{Fact}
\theoremstyle{plain}

\newcommand{\R}{\mathbb{R}}
\newcommand{\N}{\mathbb{N}}
\newcommand{\Z}{\mathbb{Z}}

\DeclareMathOperator{\LC}{VLC}

\DeclareMathOperator{\cpc}{cap}
\DeclareMathOperator{\supp}{supp}
\DeclareMathOperator{\newt}{newt}
\DeclareMathOperator{\rank}{rank}
\DeclareMathOperator{\sgn}{sgn}
\DeclareMathOperator{\conv}{conv}
\DeclareMathOperator{\Deg}{Deg}
\DeclareMathOperator{\innt}{int}
\DeclareMathOperator{\cone}{cone}
\DeclareMathOperator{\sspan}{span}

\crefname{question}{question}{questions}
\Crefname{question}{Question}{Questions}

\begin{document}
\maketitle

\begin{abstract}
    We present a new lower bound on the number of lattice points of all totally unimodular polytopes, generalizing previous lower bounds on contingency tables, integer flows, and beyond. Our bound is based on the Gurvits capacity convex optimization problem, and thus our result implies an efficient deterministic algorithm for approximate counting of the lattice points up to an explicit exponential factor. We achieve our bounds by showing that the associated generating polynomials fit into a new general class of log-concave polynomials called $\LC$ (``variable-wise log-concavity''). This also implies a conjecture of Ferroni and Higashitani on the evaluations of the Ehrhart polynomials of unimodular polytopes. The essential ingredient for these results is the resolution of Barvinok's log-concavity conjecture for contingency tables on lines, which was proven using ChatGPT 6 Astra. We conjecture a generalization of Barvinok's conjecture, which we believe will lead to stronger and more general bounds.
\end{abstract}

\section{Introduction}

Totally unimodular polytopes are integral polytopes defined by a totally unimodular constraint matrix; concretely, a convex polytope $Q$ is totally unimodular (TU) if there exists a totally unimodular matrix $M$ (i.e., all minors are in $\{-1,0,+1\}$) and an integral vector $\bm{b}$ such that $Q = \{\bm{y} \geq 0 : M\bm{y} = \bm{b}\}$. Such polytopes have a wide range of applications, with their lattice points often representing integer solutions to various combinatorial problems. These applications include job assignments, contingency tables, bipartite matchings, bipartite independent sets, network flows, routing configurations, and beyond.

In this paper, we address two natural questions for such polytopes:
\begin{enumerate}
    \item How can we estimate the number of lattice points of such a polytope $Q$?
    \item How does the number of lattice points of $Q$ vary with the choice of $\bm{b}$?
\end{enumerate}
Answers to both of these questions arise by proving log-concavity properties of generating functions associated to totally unimodular polytopes.

Given a totally unimodular polytope $Q = \{\bm{y} \geq 0 : M\bm{y} = \bm{b}\}$, there is a natural generating function for its lattice points based on the choice of $\bm{b}$. Concretely, if $M$ is $m \times n$ and has columns $\bm{a}_1,\ldots,\bm{a}_n$ then
\[
    P_M(\bm{x}) = \prod_{j=1}^n \frac{1}{1-\bm{x}^{\bm{a}_j}}
\]
is a generating function with the property that the $\bm{x}^{\bm{b}}$ coefficient of $P_M(\bm{x})$ is precisely the number of lattice points of $Q$. Both of the above questions then become questions about $P_M$: how can we estimate the coefficients of $P_M$, and how can we understand how the coefficients of $P_M$ relate to one another?

To address these questions, we prove that truncations of $P_M$ satisfy a certain natural log-concavity property called $\LC$. This property says that, after applying coefficient extractions and positive evaluations to all but one variable of $P_M$, the coefficients of the remaining polynomial are log-concave. This generalizes various notions of log-concave polynomials (strongly log-concave \cite{Gur09}, completely log-concave \cite{ALOGV24}, Lorentzian \cite{BH20}, etc.) which have had widespread applications in mathematics and computer science.

The fact that truncations of $P_M$ are $\LC$ is then based on the underlying main technical result of the paper: that the coefficients of $P_M$ are log-concave along any rational line. This says in fact that as $\bm{b}$ moves along any line, the number of lattice points of $Q$ forms a log-concave sequence. Barvinok originally asked a more general question of this form for the case of contingency tables \cite{Barvinok2007BM}, and we resolve this line version here for all TU polytopes. This gives a nice answer to question 2, while the full strength of Barvinok's question for TU polytopes is still open \Cref{open-question-LC}.

The $\LC$ property then enables us also to answer question 1. The Gurvits capacity method was originally developed to give a simple proof of the van der Waerden conjecture for permanents in \cite{Gur08}. This technique implies a deterministic simply exponential approximation algorithm for non-negative matrix permanents as well. Since this original work, the capacity method has seen numerous applications to bounding and deterministically approximating various other combinatorial quantities. The general idea is to use log-concavity properties of a given generating function to approximate its coefficients by an associated convex optimization problem. The $\LC$ property is then tailor-made precisely to allow the Gurvits capacity method to work as usual, and thus we can apply this method to $P_M$ and related generating functions.

\subsection{Our Results}

We now formally answer question 1: how can we estimate the number of lattice points of a TU polytope? We give two results in this direction. The first gives general bounds for all TU polytopes, and the second gives simply exponential bounds in a slightly restricted regime.

\begin{theorem} [\Cref{simply exponential bound using a linear transformation} and \Cref{sec:tightness}]\label{main-thm}
    Let $M$ be an $m \times n$ totally unimodular matrix with columns $\bm{a}_1,\ldots,\bm{a}_n$. For all $\bm{b} \in \Z^m$ such that $\#\{\bm{y} \in \Z_{\geq 0}^n \mid M\bm{y} = \bm{b}\} < \infty$, we have
    \[
         \mathcal{E} \geq \#\{\bm{y} \in \Z_{\geq 0}^n \mid M\bm{y} = \bm{b}\} \geq e^{-m} (1+\|\bm{b}\|_1)^{-m} \mathcal{E},
    \]
    where
    \[
        \mathcal{E} = \inf_{ \substack{ \bm {0} < \bm x \\ \bm x^{\bm a_i} < 1 \, \forall i \in [n]}} \bm{x}^{-\bm{b}} \prod_{j=1}^n \frac{1 }{1 - \bm{x}^{\bm{a}_j}} = \inf_{\bm{x} > 0} \bm{x}^{-\bm{b}} \prod_{j=1}^n \sum_{k\ge 0} \bm{x}^{k \bm{a}_j}
    \]
    is a convex program up to $\log$-$\log$ transformation. The lower bound has tight dependence on $\|\bm{b}\|_1$.
\end{theorem}

\begin{corollary}[\Cref{simply exponential bound using a linear transformation}] \label{main-cor}
    Let $M$ be an $m \times n$ totally unimodular matrix with columns $\bm{a}_1,\ldots,\bm{a}_n$ and $m \ge 2$. For all $\bm{b} \in \Z^m$ satisfying $\norm{\bm{b}}_1 \ge 2m \cdot \log(2m)$ and such that $\#\{\bm{y} \in \Z_{\geq 0}^n \mid M\bm{y} = \bm{b}\} < \infty$, we have
    \[
         \mathcal{E} \geq \#\{\bm{y} \in \Z_{\geq 0}^n \mid M\bm{y} = \bm{b}\} \geq e^{-\|\bm{b}\|_1} \mathcal{E},
    \]
    where $\mathcal{E}$ is as in \Cref{main-thm}.
\end{corollary}

\Cref{main-thm} and \Cref{main-cor} can also be made algorithmic, using the maximum-entropy optimization framework of
\cite{SV14}.

\begin{theorem}[see \Cref{efficiency-proof}] \label{main-algo}
    There is a deterministic algorithm (based on the ellipsoid algorithm) which, given an $m \times n$ totally unimodular matrix $M$, a vector $\bm{b} \in \Z^m$ such that $\{\bm{y} \in \Z_{\geq 0}^n \mid M\bm{y} = \bm{b}\}$ is nonempty, and an $\epsilon \in (0,1)$, computes $\mathcal{A}$ such that $\mathcal{A} - \epsilon \leq \log \mathcal{E}' \leq \mathcal{A} + \epsilon$ in time polynomial in $m$, $n$, $\log(1+\|\bm{b}\|_1)$, and $\log \frac{1}{\epsilon}$. Here $\mathcal{E}'$ is a certain polynomial truncation of the $\mathcal{E}$ defined in \Cref{main-thm}.
\end{theorem}

The proofs of \Cref{main-thm} and \Cref{main-cor} are then based on our answer to question 2: how does the number of lattice points of a TU polytope vary with $\bm{b}$? The following result answers this by generalizing a question of Barvinok \cite{Barvinok2007BM} on the log-concavity of the counts of contingency tables with various marginals. Our result in fact holds for \textbf{unimodular polytopes}, which are defined like totally unimodular polytopes except that the defining matrix $M$ is only required to be unimodular.

\begin{theorem} [\Cref{thm- Midpoint log-concavity for TU matrices}] \label{main-Barvinok}
    Let $M$ be an $m \times n$ (not necessarily totally) unimodular matrix. For all $\bm{b},\bm{b}^-,\bm{b}^+ \in \Z^m$ such that $\bm{b} = \frac{\bm{b}^- + \bm{b}^+}{2}$, we have
    \[
        \#\{\bm{y} \in \Z_{\geq 0}^n \mid M\bm{y} = \bm{b}\}^2 \geq \#\{\bm{y} \in \Z_{\geq 0}^n \mid M\bm{y} = \bm{b}^-\} \cdot \#\{\bm{y} \in \Z_{\geq 0}^n \mid M\bm{y} = \bm{b}^+\},
       \]
    given that all of these fibers are finite. 
\end{theorem}

We note that \Cref{main-Barvinok} immediately implies a recent conjecture for TU polytopes on the log-concavity of Ehrhart evaluations. This conjecture comes from \cite{FerroniHigashitani2024}, where the authors in fact make the conjecture for all IDP lattice polytopes (though this is now false in general by \cite{Fer26}).
Recall that, given a lattice polytope $P$, the Ehrhart polynomial $E_P(t)$ is the unique polynomial such that $E_P(k) = \#\left(kP \cap \Z^n\right)$ for all $k \in \Z_{\geq 0}$.

\begin{corollary}[Conjecture 1.2 of \cite{FerroniHigashitani2024}] \label{cor:ehrhart}
    For all (not necessarily totally) unimodular polytopes $P$, we have $E_P(k)^2 \geq E_P(k-1) \cdot E_P(k+1)$ for all $k \in \Z_{\geq 1}$.
\end{corollary}

\subsection{An Open Question}

\Cref{main-Barvinok} is in fact based upon a stronger coefficient-wise positivity result (\Cref{thm- Midpoint log-concavity for TU matrices}) which follows from the bijective proof which we give in \Cref{section: TU matrices}. This theorem raises an open question, which is the natural generalization of Barvinok's question in \cite{Barvinok2007BM}.

\begin{question} \label{open-question-LC}
    Let $M \in \Z^{m \times n}$ be a unimodular matrix and $\mu$ a probability distribution on $\Z^m$ (perhaps with finite support) such that its mean is $\bm{b} \in \Z^m$. Is it that case that
    \[
        \#\{\bm{y} \in \Z_{\geq 0}^n \mid M\bm{y} = \bm{b}\} \geq \prod_{\bm{a} \in \Z^m} \#\{\bm{y} \in \Z_{\geq 0}^n \mid M\bm{y} = \bm{a}\}^{\mu(\bm{a})}?
    \]
\end{question}

Note that for functions/distributions over $\R^m$, \Cref{open-question-LC} would immediately follow from \Cref{main-Barvinok}. However, for functions/distributions over $\Z^m$, convexity on lines does not imply global convexity in general.

We believe that a positive answer to \Cref{open-question-LC} would enable us to obtain stronger capacity bounds (hopefully simply exponential) in general, using the following idea. Currently, most capacity bounds are proved inductively on each variable of a given polynomial or generating series. Because of this, the bounds often depend on distances of the desired marginals or degree vector (e.g. the $\bm{b}$ vector) to the boundary of the Newton polyhedron in the coordinate directions. The coordinate direction distance can be very bad for general TU polytopes, and so we need to adapt the application of the capacity bounds to handle this. Given the inequalities of \Cref{open-question-LC}, it is perhaps possible that this capacity induction argument may be able to be exchanged with a geometric argument related to distances to the boundary of the Newton polyhedron in any choice of directions. This would yield much better bounds and approximation factors, if possible.



\subsection{Previous and Related Work}

\paragraph{Approximation.}

Bounding and approximately counting lattice points in totally unimodular
polytopes has attracted considerable attention. A variety of bounds and
approximation algorithms, both randomized and deterministic, are known
for important classes of such polytopes. On the randomized side, fully
polynomial randomized approximation schemes (FPRASs) have been obtained
for bipartite perfect matchings \cite{JSV04}, binary contingency tables
\cite{BBV07}, and nonnegative integer contingency tables under restrictions
such as a fixed number of rows or sufficiently large margins
\cite{CD03,DKM97,M02}. Related results apply to cell-bounded contingency
tables with a fixed number of rows and to integral flows with sufficiently
large tight capacities \cite{CDR10}. For certain classes of smooth margins,
a randomized approximation algorithm runs in time quasipolynomial in the
total sum of the table entries \cite{BLSY10}. More recently, rapid mixing
of the natural swap chain was established for binary contingency tables
with arbitrary feasible margins \cite{FQW26}.

On the deterministic side, matrix scaling and polynomial capacity methods
give polynomial-time exponential approximations for counting bipartite
perfect matchings \cite{LSW00,Gur06,Gur08}, with subsequent improvements
\cite{GS14,AR25}. Recent preprints obtain a subexponential approximation
factor \cite{KLM26} and a $(1+\varepsilon)^n$ approximation factor for every
fixed $\varepsilon>0$ \cite{DJ26}, where $n$ is the number of vertices on
each side of the bipartition. Convex optimization and polynomial capacity
methods also give (often simply) exponential estimates for contingency tables
\cite{Bar09,Bar10,Gur15,BLP23} and integer flows \cite{Bar09,LM26,LMY26}.
Computationally efficient asymptotic formulas are available in suitable
smooth-margin regimes \cite{BH12}. For contingency tables with a fixed
number of rows, deterministic fully polynomial-time approximation schemes
(FPTASs) are known \cite{GKMSVV11}, with improved running times in the
two-row case \cite{AH21}.

Beyond these special classes, other positive results include polynomial-time
exact counting in fixed dimension \cite{Bar94}, exact counting methods
for nonnegative totally unimodular systems with a fixed number of equations
\cite{Mount00}, algebraic methods for computing counting
polynomials for unimodular systems \cite{DLS03}, and
maximum-entropy Gaussian approximations under well-roundedness assumptions \cite{BH10}.

In contrast, no fully polynomial randomized or deterministic
approximation scheme is known for arbitrary totally unimodular polytopes; such an
algorithm would in particular yield an FPRAS for counting
independent sets in bipartite graphs \cite{GJ23}.
This motivates seeking weaker approximation guarantees that apply uniformly to all bounded TU polytopes, which is what we prove here.



\paragraph{Log-concavity.}

The Ehrhart log-concavity statement in \Cref{cor:ehrhart} has several antecedents. For a finite poset $Q$, the Ehrhart polynomial of its order polytope satisfies $E_{\mathcal O(Q)}(k)=\Omega(Q,k+1)$, where $\Omega(Q,t)$ counts weakly order-preserving maps from $Q$ to $[t]$. Log-concavity of these evaluations was proved in \cite[Theorem~5]{DDP84} and by Brenti \cite[Theorem~7.6.5]{Bre89}; see also \cite[Section~6.1]{FH24}. \cite{CPP23} gives an alternative proof using the FKG inequality, further obtaining a quantitative strengthening to strict log-concavity. A multivariate coefficientwise refinement was further established by \cite[Corollary~9.5]{CP23}, providing a precedent for the coefficientwise positivity statement in \Cref{thm- Midpoint log-concavity for TU matrices}.

Another sufficient condition for log-concavity of Ehrhart evaluations is log-concavity of the $h^*$-polynomial \cite[Proposition~5.3]{FH24}. This follows by expressing the Ehrhart sequence as the convolution of the $h^*$-vector with the log-concave sequence $\binom{k+d}{d}$, where $d$ is the dimension. In particular, this criterion applies to lattice zonotopes, whose $h^*$-polynomials are real-rooted by \cite[Theorem~1.2]{BJM19}, and to Pitman--Stanley polytopes with positive integer parameters by \cite[Theorems~1.1--1.2]{AFM26}. Log-concavity of Ehrhart evaluations for all lattice polygons was also proven in \cite[Theorem~5.9]{FH24}.

Our results establish log-concavity of Ehrhart evaluations uniformly for polytopes defined by unimodular systems, without an additional assumption on their $h^*$-polynomials. More generally, \Cref{thm- Midpoint log-concavity for TU matrices} gives coefficientwise midpoint log-concavity for lattice-point generating polynomials of unimodular fibers with fixed entrywise bounds, as the right-hand side varies.



\subsection{AI Declaration}

The AI tool we specifically used was ChatGPT 6 Astra. The authors have verified and take full responsibility for everything written in the paper.

\paragraph{Before the first draft.} The resolution of Barvinok's conjecture/question on the log-concavity on lines of the contingency tables counts (\Cref{main-Barvinok}) was the main AI input to this paper. We generalized this result to all TU polytopes with upper and lower entrywise bounds on the counted lattice points (\Cref{thm- Midpoint log-concavity for TU matrices}), and in fact the AI suggested in a remark that this should be possible. Beyond that, the literature review and much of the associated parts of the introduction were written with the help of AI.

\paragraph{After the first draft.} We used AI to check for errors or to make suggestions on the flow or style of the paper. From this we fixed a number of errors which were mostly minor; the main issue that was fixed here were proofs regarding limiting arguments for capacity and for sequences of Laurent polynomials and series. In particular \Cref{cap of PMk converges to cap PM} and the surrounding results were added at this point. 

We also originally had referenced \cite{SV14,SV19} as the way to prove algorithmic efficiency in \Cref{main-algo}, but we had left out the details. We used AI to fill in those details, and now we have the proof of \Cref{main-algo} in \Cref{efficiency-proof}. Further, AI also found that the central-fiber argument in our (AI-assisted) proof of \Cref{main-Barvinok} has a direct antecedent in the graph-orientation argument of \cite{CI20}. Finally, AI also suggested many changes to the structure, ordering, and content of the introduction and title. We took many of those suggestions, and then wrote the new introduction ourselves.

\paragraph{After the first arXiv submission. } Instead of the approach taken in \Cref{initial lower bound for TU matrices using a lin transformation}, we had initially used the bound presented at the beginning of \Cref{section: capacity bounds} for a specific large vector $\bm k$ to obtain a bound for $B_{M}(\bm b)$. We later had the idea of studying the geometry of the cone generated by $\supp(P_{M,\bm k})$ to obtain degree bounds $\beta_i(\bm k)$ independent of $\bm k$. We were specifically looking for a series of such linear transformations that preserve total unimodularity, fiber counts and capacity. We used AI find a suitable linear transformation. AI also helped us recognize that a single transformation suffices. Finally, AI also was used to come up with the tight example for \Cref{main-thm}. The authors verified AI's construction and arguments.

\section{$\LC$ Laurent polynomials}
\label{section: LC polynomials}

Gurvits' capacity function has been widely used to obtain lower bounds for various interesting combinatorial quantities.
The underlying idea of these results is that if a (Laurent) polynomial has a certain log-concavity property, then one could use the theory of log-concave polynomials to derive lower bounds for its coefficients. However, currently studied log-concavity properties (e.g., Lorentzian, denormalized/dually Lorentzian, etc.) are all conditions which are too strong to apply to TU polytopes in general.

Thus, in \Cref{subsection: definition and capacity bounds} we introduce a class of Laurent polynomials with weaker log-concavity properties, for which capacity lower bounds are still applicable. We further explore basic properties of this class throughout \Cref{subsection: examples and properties}. We then turn our focus to unimodular and totally unimodular (TU) matrices in \Cref{section: TU matrices} and prove that the cardinalities of the fibers of a unimodular polytope are log-concave along lines.

\subsection{Definition and capacity bounds}
\label{subsection: definition and capacity bounds}

There are two key steps in proving capacity bounds for a denormalized Lorentzian polynomial $p$: first that any evaluation of $p$ in all but one variable is a univariate polynomial with log-concave coefficients, and second that $[x_i^{\alpha_i}]p$ satisfies the first condition for any $i, \alpha_i$. We will take these two properties and define a new class of Laurent polynomials. Capacity lower bounds then naturally extend to this class.

Recall that a sequence $\{a_k\}_{k \in \Z}$ of non-negative numbers is said to be log-concave if it has no internal zeros and $a_k^2 \ge a_{k-1} \cdot a_{k+1}$ for any valid index $k$. A Laurent polynomial is a sum $p(\bm{x}) = \sum_{\bm{\alpha}\in \Z^n} p_{\bm{\alpha}}\bm{x^{\alpha}}$ with a finite support, where the support of $p$ denoted by $\supp(p)$ is the set $\{\bm{\alpha} \in \Z^n: p_{\bm{\alpha}} \neq 0 \}$. We further say that a univariate Laurent polynomial $p(x) = \sum_{k \in \Z} p_kx^k$ has a log-concave coefficient sequence if $\{ p_k\}_{k \in \Z}$ is a log-concave sequence. 

\begin{definition}
\label{definition: LC polynomials}
Define: 
\[
\LC_1 = \{ p(x) \in \R[x^{\pm1}] : p \text{ has a log-concave coefficient sequence} \}. 
\]
and for any $n \ge 2$, let $\LC_n$ contain all $n$-variate Laurent polynomials $p(x_1, \cdots, x_n) = \sum_{\bm{\alpha} \in \Z^n}p_{\bm{\alpha}}\bm{x^{\alpha}}$ with non-negative coefficients such that:
\begin{itemize}
    \item Evaluating $p$ at any $n-1$ variables with positive real numbers belongs to $\LC_1$, 
    \item $[x_i^{\alpha_i}] p$ is in $\LC_{n-1}$ for any $i \in [n]$ and $\alpha_i \in \Z$.
\end{itemize}    

Equivalently, $p(x_1, \cdots, x_n) \in \R_{\ge 0}[\bm{x},\bm{x}^{-1}]$ is an element of $\LC_n$ if for any $S \subseteq [n]$ and any $\bm{\alpha} \in \Z^S$, any positive evaluation of $[\prod_{i \in S}{x}_i^{{\alpha}_i}] p$ in all but one variable has log-concave coefficients. In a sense, the coefficients of $p$ are log-concave after applying certain variable-wise operations, hence the name $\LC$.

More generally, we will say that a polynomial $p(x_1, \cdots,x_n)$ is $\LC$ if it belongs to $\LC_n$. 
\end{definition}

Lorentzian polynomials and denormalized Lorentzian polynomials are both subclasses of $\LC$ (see \cite[\S 3.1]{BH20}). In fact, this class of Laurent polynomials is crafted in a way that the capacity bound for the coefficients of denormalized Lorentzian Polynomials extends naturally to $\LC$ Laurent polynomials.

\begin{proposition}
\label{capacity bound for LC polynomials}
    Suppose $p (x_1, \cdots, x_n) \in \LC_n$ and $\bm{\alpha} \in \Z^n$, then: 
    \[
    [\bm{x^{\alpha}}] p(\bm{x}) \ge \cpc_{\bm{\alpha}}(p) \cdot \prod_{i = 1}^n \frac{|\beta_i - \alpha_i|^{|\beta_i - \alpha_i|}}{(1 + |\beta_i - \alpha_i|)^{1 + |\beta_i - \alpha_i|}}
    \]
    where $\beta_i$ is either a lower bound or an upper bound for $\Deg_{i}^{\bm{\alpha}}(p)$, the set of degrees of $x_i$ appearing in $[{x_{i+1}^{\alpha_{i + 1}}} \cdots x_n^{\alpha_n}] p$.
\end{proposition}
\begin{proof}
    The proof is done by induction over the number of variables of $p$.  For $n = 1$, the proof follows from the case for denormalized Lorentzian Laurent polynomials by \cite[Lemma 4.20]{LMY26} (see also \cite[Lemma 5.7]{BLP23}). Note that the homogenization of the polynomial $\sum_{k \in \Z} a_k x^k$ is denormalized Lorentzian iff $\{ a_k\}_{k \in \Z}$ is a log-concave sequence, i.e. $\sum_{k \in \Z} a_k x^k$ is $\LC$. 
    
    Suppose $n \ge 2$. Then for any $(r_1, \cdots, r_{n-1}) \in \R_{> 0}^{n-1}$: 
    \begin{align*}
        \cpc_{\bm{\alpha}}(p) &\le \inf_{x_n >0 }\frac{p(r_1, \cdots, r_{n-1}, x_n)}{r_1^{\alpha_1} \cdots r_{n-1}^{\alpha_{n-1}}x_n^{\alpha_n}} \\
        & = \frac{1}{r_1^{\alpha_1} \cdots r_{n-1}^{\alpha_{n-1}}} \inf_{x_n > 0} \frac{p(r_1, \cdots, r_{n-1}, x_n)}{x_n^{\alpha_n}} \\
        & = \frac{1}{r_1^{\alpha_1} \cdots r_{n-1}^{\alpha_{n-1}}}\cpc_{\alpha_n} p(r_1, \cdots, r_{n-1}, x_n) \\
        & \le \frac{1}{r_1^{\alpha_1} \cdots r_{n-1}^{\alpha_{n-1}}}\cdot \frac{(1 + |\beta_n - \alpha_n|)^{1 + |\beta_n - \alpha_n|}}{|\beta_n - \alpha_n|^{|\beta_n - \alpha_n|}} \cdot [x_n^{\alpha_n}] p(r_1, \cdots, r_{n-1}, x_n)
    \end{align*}
    Since this is true for any $(r_1, \cdots, r_{n-1})$, we conclude: 
    \begin{align*}
        \cpc_{\bm{\alpha}}(p) &\le \frac{(1 + |\beta_n - \alpha_n|)^{1 + |\beta_n - \alpha_n|}}{|\beta_n - \alpha_n|^{|\beta_n - \alpha_n|}} \cdot \inf_{r_1, \cdots, r_{n-1} > 0}\frac{[x_n^{\alpha_n}] p(r_1, \cdots, r_{n-1}, x_n)}{r_1^{\alpha_1} \cdots r_{n-1}^{\alpha_{n-1}}} \\ &=\frac{(1 + |\beta_n - \alpha_n|)^{1 + |\beta_n - \alpha_n|}}{|\beta_n - \alpha_n|^{|\beta_n - \alpha_n|}} \cdot \cpc_{(\alpha_1, \cdots, \alpha_{n-1})}([x_n^{\alpha_n}] p).
    \end{align*}
    The proof is then completed by the induction hypothesis for $[x_n^{\alpha_n}]p$.   
    \end{proof}

\subsection{Examples and properties}
\label{subsection: examples and properties}
As mentioned before, any denormalized Lorentzian polynomial is $\LC$. It is easy to see that any ``shift" of a denormalized Lorentzian polynomial is also $\LC$, i.e. $\bm{x^{\alpha}}p(\bm{x})$ is $\LC$ if $p$ is denormalized Lorentzian. However, the class of $\LC$ Laurent polynomials contains a much larger set of Laurent polynomials. The following is an example of a Laurent polynomial that is not necessarily denormalized Lorentzian but is $\LC$:
\begin{example}
\label{example: a small example}
    Let $n,k \in \Z_{\ge 0}$ and $\bm{b} \in \{0, \pm 1 \}^n$, then:
   \[
   p(x_1, \cdots, x_n) = \sum_{t = 0}^k \bm{x}^{t \cdot \bm{b}} 
   \]
   is $\LC$. We proceed by induction on $n$. If $n = 1$, the statement is trivial since the coefficients sequence of $p$ is the all ones sequence. Now suppose $n \ge 2$ and assume the statement holds for $n-1$. Let us check both of the conditions of \Cref{definition: LC polynomials} for $p$. 
   
   First evaluate $p$ at all but one variable, say $x_{i}$. It is easy to check that the coefficients of the polynomial $\restr{p(\bm{x})}{x_j = r_j \forall j \neq i}$ form a geometric sequence and thus are log-concave. 
   Furthermore, if  $b_i \neq 0$, the Laurent polynomial $[x_i^{\alpha_i}] p(\bm{x})$ is either the zero polynomial or just a monomial and is $\LC$. If $b_i = 0$ however, $[x_i^{\alpha_i}]p$ is really a Laurent polynomial in $n-1$ variables which is $\LC$ by the induction hypothesis. 

    As another example, suppose $\bm{b}_1, \bm{b}_2 \in \{0, 1 \}^n$, then the series: \[q(\bm{x}) = \sum_{t = 0}^k \bm{x}^{t\cdot \bm{b}_1}  \cdot \sum_{t = 0}^k \bm{x}^{t \cdot \bm{b}_2}\] 
    can also be shown to be $\LC$ by basic algebra and some extra effort. We have refrained from including the details of this proof for brevity. One could see that this polynomial is indeed $\LC$ by \Cref{pAB is in LC}. 
\end{example}

For general $\bm{b}, \bm{b}_1, \bm{b}_2$, the supports of the Laurent polynomials $p$ and $q$ in \Cref{example: a small example} need not be $M$-convex. Thus, neither is necessarily a denormalized Lorentzian Laurent polynomial. This raises the question of characterizing the support of an $\LC$ Laurent polynomial.  

\begin{proposition}
    \label{support is saturated for LC polynomials}
    If $p(\bm{x}) = \sum_{\bm{\alpha} \in \Z^n} p_{\bm{\alpha}}\bm{x^{\alpha}}$ is an $\LC$ Laurent polynomial, then its support is a finite saturated set. 

    $S \subseteq \Z^n$ is called a saturated set if $S = \conv(S) \cap \Z^n$, where $\conv(S)$ is the convex hull of $S$. 
\end{proposition}
Before proving this proposition, let us recall a well-known fact about $\newt(p) := \conv(\supp(p))$, also known as the Newton polytope of $p$. 
\begin{fact}[{\cite[Fact 2.18]{AO17}, see also \cite[Lemma 4.15]{LMY26}}]
\label{capacity is nonzero iff alpha is in newton polytope}
    If $p$ is a nonzero Laurent polynomial (or a Laurent series with a nonempty domain of convergence and a closed Newton polyhedron) with nonnegative coefficients, then $\cpc_{\bm{\alpha}}(p) > 0 $ iff $\alpha \in \newt(p)$. 
\end{fact}

\begin{proof}[Proof of \Cref{support is saturated for LC polynomials}]
    It suffices to show that $p_{\bm{\alpha}}> 0$ for any $\bm{\alpha} \in \newt(p) \cap \Z^n$. By \Cref{capacity is nonzero iff alpha is in newton polytope}, for any such $\bm{\alpha}$, $\cpc_{\bm{\alpha}}(p) > 0$. Thus, by \Cref{capacity bound for LC polynomials}, we have:
    \[
    p_{\bm{\alpha}} \ge \cpc_{\bm{\alpha}}(p) \cdot \prod_{i = 1}^n \frac{|\beta_i - \alpha_i|^{|\beta_i - \alpha_i|}}{(1 + |\beta_i - \alpha_i|)^{1 + |\beta_i - \alpha_i|}} > 0.
    \]
\end{proof}

One would hope that any finite saturated set would be the support of some $\LC$ Laurent polynomial, but this is not the case. For instance, the set $S = \{(0,0), (1,0), (2,1) \}$ is a saturated set, but the polynomial $p(x,y) = a + bx + cx^2y$ is not $\LC$ no matter the choice of $a,b,c > 0$. This is because evaluating $y$ at a large enough $r > 0$ breaks log-concavity of the sequence $a,b,cr$. We currently do not have a necessary and sufficient condition for a set $S$ to arise the support of an $\LC$ Laurent polynomial. 

\begin{proposition}
    The class of $\LC$ Laurent polynomials is closed under scaling variables. 
\end{proposition}
\begin{proof}
We proceed by induction on the number of variables. If $p(x) = \sum_{k \in \Z}p_kx^k$ has a log-concave coefficient sequence, then so does $p(cx) = \sum_{k \in \Z}c^kp_kx^k$, and the induction basis holds. 

    Let $n \ge 2$ and let $p(x_1, \cdots, x_n)$ be an $\LC$ Laurent polynomial and let $c>0$. We want to show that $q(\bm{x}) = p(cx_1, x_2 , \cdots, x_n)$ is also an $\LC$ Laurent polynomial. Any evaluation of $q$ at all but the variable $x_i$ where $i \neq 1$ has log-concave coefficients, since:
    \[
    \restr{q(\bm{x})}{x_j = r_j, j \neq i} = \restr{p(\bm{x})}{x_1 = cr_1, x_j = r_j j \neq i,1}.
    \]
    If $q$ is evaluated at all but variables except $x_1$, then:
    \[q(x_1, r_2, \cdots, r_n) = p(cx_1, r_2, \cdots, x_n) = \sum_{k  \in \Z} p_k c^kx_1^k,\]
    where $p_k$ is the coefficient of $x_1^k$ in $p(x_1, r_2, \cdots, r_n)$. Since $p$ is $\LC$, the sequence $\{ p_k\}_{k \in \Z} $ is log-concave, and therefore, so is the sequence $\{ c^kp_k\}_{k \in \Z}$. Hence any evaluation of $q$ in $n-1$ variables has log-concave coefficients. 

    Furthermore, we know that $[x_i^{\alpha}]p$ is an $\LC$ Laurent polynomial in $n-1$ variables for any $\alpha \in \Z$ and $i \in [n]$. By the induction hypothesis, for any $i \neq 1$, the Laurent polynomial $[x_i^{\alpha}]q = ([x_i^{\alpha}]p)(cx_1, \cdots, x_{i-1}, x_{i+1}, \cdots ,x_n)$ is $\LC$ as well. Lastly, $[x_1^{\alpha}] q = c^{\alpha} [x_1^{\alpha}]p$ is $\LC$ too and and we are done. 
\end{proof}
\begin{proposition}
    If a sequence $\{ p_i(\bm{x})\}_{i \ge 0}$ of $\LC_n$ Laurent polynomials converges coefficient-wise to a Laurent polynomial $p(\bm{x})$ such that $\supp(p_i) \subseteq \supp(p)$ for all $i \ge 0$, then $p \in \LC_n$. 
\end{proposition}
\begin{proof}
    Proceed with induction on $n$. The statement for $n = 1$ follows from the fact that the limit of any log-concave sequence is a log-concave sequence. 
    Let $n \ge 2$. Suppose $r_2, \cdots, r_n >0$. Then: 
    \[
    p(x_1, r_2, \cdots, r_n) = \lim_{i \to \infty} p_i(x_1, r_2, \cdots, r_n),
    \]
    since each $p_i(x_1, r_2, \cdots, r_n)$ has a log-concave coefficient sequence, the coefficient sequence of $p(x_1, r_2, \cdots, r_n)$ is log-concave as well. Moreover, for any $\alpha \in \Z$, $j \in [n]$: 
    \[
    [x_j^{\alpha}] p = \lim_{i \to \infty} [x_j^{\alpha}] p_i.
    \]
     Since $\supp([x_j^{\alpha}]p_i) \subseteq \supp([x_j^{\alpha}]p)$ for all $i \ge 0$, $[x_j^{\alpha}] p$ is $\LC_{n-1}$ by the induction hypothesis, finishing the proof.
\end{proof}
\section{Unimodular and Totally unimodular (TU) polytopes}
\label{section: TU matrices}

The goal of this section is to prove the following log-concavity statement: \[ \left(\sum_{ \substack{M\bm{y} =\frac{\bm{b}^+ + \bm{b}^- }{2}\\ \bm k_1 \le \bm y  \le \bm k_2 }} \bm{w^y} \right)^2 \ge \sum_{ \substack{M\bm{y} ={\bm{b}^+}\\ \bm k_1 \le \bm y  \le \bm k_2  }} \bm{w^y} \cdot \sum_{ \substack{M\bm{y} ={\bm{b}^-}\\ \bm k_1 \le \bm y  \le \bm k_2  }} \bm{w^y},\]
where $M$ is an $m \times n$ unimodular matrix and $\bm{b}^+, \bm b^-, \frac{\bm{b}^+ + \bm b^-}{2}\in \Z^m, \bm{w} \in \R_{>0}^n$ are given vectors. Recall that a matrix $M$ is said to be unimodular if all of its non-singular $\rank(M) \times \rank(M)$ submatrices have determinant $\pm 1$. We will provide an AI-assisted proof that this sequence is log-concave for any $M$ and $\bm{b}$ in \Cref{subsection-midpoint log-concavity for tu}. In \Cref{subsection: solution Laurent polynomials are LC}, we will use this log-concavity statement to show that $\sum_{\bm 0 \le \bm{y} \le \bm{k}} \bm{x}^{M \bm{y}}$ is an $\LC$ Laurent polynomial for any TU matrix $M$. We will then derive capacity bounds for $\# \{\bm{y} \in \Z_{\ge 0}^n: M \bm{y} = \bm{b} , \bm{y} \le \bm{k}\}$ using this Laurent polynomial in \Cref{section: capacity bounds}. 

\subsection{Mid point log-concavity for unimodular polytopes}
\label{subsection-midpoint log-concavity for tu}
The ideas and proofs in this section follow an AI assisted argument for midpoint log-concavity of contingency tables with tweaks to prove what we require for our purposes. In particular, the proofs of \Cref{thm- Midpoint log-concavity for TU matrices}, \Cref{B^2 in terms of E and B} and \Cref{midpoint is heavy} are AI-assisted. We note that the pair-decomposition and induction used in \Cref{B^2 in terms of E and B} and \Cref{midpoint is heavy} generalize the argument from Section~5 of \cite{CI20} from graph orientations to lattice points in boxes subject to unimodular linear constraints.

For any unimodular matrix $M \in \Z^{m \times n}$ and any vectors $\bm{b}, \bm{a} , \bm{k}_1, \bm{k}_2$ of proper dimensions, Define:
\begin{align*}
    \mathcal{B}_{(\bm{k}_1, \bm{k}_2)}(M,\bm{b}) &= \{ \bm{y} :  M \bm{y} = \bm{b}, \bm{k}_1 \le \bm{y} \le \bm{k}_2\},\\
    B_{(\bm{k}_1, \bm{k}_2)}(M,\bm{b}) &= |\mathcal{B}_{(\bm{k}_1, \bm{k}_2)}(\bm{b}) \cap \Z^n|,\\
    \mathcal{E}(M,\bm{a}) &= \{ \bm{y}: M \bm{y} = \frac{M\bm{a}}{2}, y_i \in \{0,a_i\} \forall i \in [n]\},\\
    E(M,\bm{a}) &= |\mathcal{E}(M,\bm{a})|,\\
    T(M,\bm{b}, \bm{k}_1, \bm{k}_2; \bm{w}) & = \sum_{\substack{\bm{k}_1 \le \bm{y} \le \bm{k}_2 \\ M \bm{y} = \bm{b}}} \bm{w}^{\bm{y}}.
\end{align*}
When the choice of the matrix \(M\) is clear from context, we will omit \(M\) from the notation in all of the definitions above. Note that if $E(M , \bm{a}) \neq 0$, then we must have ${M\bm{a}} \in (2 \mathbb{Z})^m $. We aim to prove the following theorem: 
\begin{theorem}
\label{thm- Midpoint log-concavity for TU matrices}
    Suppose $\bm{b,b^+, b^-}$ are integer vectors such that $\bm{b} = \frac{\bm{b^+} + \bm{b^-}}{2}$. Then for any $\bm{k}_1, \bm{k}_2$:
    \[
    T(\bm{b}, \bm{k}_1, \bm{k}_2; \bm{w})^2 - T(\bm{b^+}, \bm{k}_1, \bm{k}_2; \bm{w}) \cdot T(\bm{b^-}, \bm{k}_1, \bm{k}_2; \bm{w})
    \]
    is a Laurent polynomial in $\bm{w}$ with nonnegative coefficients. 
\end{theorem}

In fact, we will explicitly express the coefficients of $T(\bm{b}, \bm{k}_1, \bm{k}_2; \bm{w})^2 - T(\bm{b^+}, \bm{k}_1, \bm{k}_2; \bm{w}) \cdot T(\bm{b^-}, \bm{k}_1, \bm{k}_2; \bm{w})$ in terms of $B_{( \bm{k}_1, \bm{k}_2)} (\bm{b})$, $B_{( \bm{k}_1, \bm{k}_2)} (\bm{b}^+)$ and $B_{( \bm{k}_1, \bm{k}_2)} (\bm{b}^-)$ and prove that these coefficients are non-negative using \Cref{B^2 in terms of E and B}, \Cref{midpoint is heavy} and \Cref{unimodular polytopes have integer vertices}. We then use \Cref{thm- Midpoint log-concavity for TU matrices} to prove that the Ehrhart of a unimodular polytope has log-concave evaluations, see \Cref{Ehrhart polynomials of unimodular polytopes are log-concave}. We explain why our methods do not work for a general IDP polytope in \Cref{why our methods dont extend to IDP polytopes}. 

\begin{lemma}
\label{B^2 in terms of E and B}
For any integer vector $\bm{b} \in \Z^m$, we have: 

\[
B_{(\bm{k}_1,\bm{k}_2)}(\bm{b})^2 = \sum_{\substack{\bm a \in \Z^n\\\bm{0} \le \bm{a} \le \bm{k}_2 - \bm{k}_1, \\  E(\bm{a}) > 0}} E(\bm{a}) B_{(\bm{k}_1, \bm{k}_2 - \bm{a})}\left(\bm{b} - \frac{M\bm{a}}{2}\right).
\]
\end{lemma}
\begin{proof}
    It suffices to find a bijection:
    \[
    \left[\mathcal{B}_{(\bm{k}_1, \bm{k}_2)}(\bm{b}) \times \mathcal{B}_{(\bm{k}_1, \bm{k}_2)}(\bm{b}) \right] \cap \Z^{2n}  \overset{\varphi}{\longrightarrow} \left\{ (\bm{a}, \bm{p}, \bm{m}) : \substack{\bm{0} \le \bm{a} \le \bm{k}_2 - \bm{k}_1, \, \bm a \in \Z^n \\ \bm{p} \in \mathcal{E}(\bm{a}) \\  \bm{m} \in \mathcal{B}_{(\bm{k}_1, \bm{k}_2 - \bm{a})}\left(\bm{b} - \frac{M\bm{a}}{2}\right) \cap \Z^n} \right\} .
    \]
    For any $\bm{x,y} \in \mathcal{B}_{(\bm{k}_1, \bm{k}_2)}(\bm{b}) \cap \Z^n$, define: 
    \[
    \bm{a} = |\bm{x} - \bm{y}|, \quad \bm{p} = (\bm{x}- \bm{y})_+, \quad \bm{m} = \min (\bm{x},\bm{y})
    \]
    where for any number $r \in \R$, we define $(r)_+ = 0$ if $r < 0$ and $(r)_+ = r$ otherwise. The following facts are easy to check: 
    \begin{enumerate}
        \item $\bm{x} = \bm{m} + \bm{p}$,
        \item $\bm{y} = \bm{m} + \bm{a} - \bm{p}$,
        \item $2\bm{p} = \bm{a} + \bm{x} - \bm{y}$ so $M \bm{p} = \frac{M\bm{a}}{2}$ and further $p_i \in \{0, a_i\}$, thus $\bm{p} \in \mathcal{E}(\bm{a})$,
        \item $2 \bm{m} = \bm{x} + \bm{y} - \bm{a}$, so $M \bm{m}  = \bm{b} - \frac{M\bm{a}}{2}$. Moreover $\bm{k}_1 \le \bm{m} \le \bm{k}_2 - \bm{a}$. Thus $\bm{m} \in \mathcal{B}_{(\bm{k}_1, \bm{k}_2 - \bm{a}) } \left(\bm{b} - \frac{M\bm{a}}{2} \right) \cap \Z^n$.
     \end{enumerate}
     Let:
    $
     \varphi(\bm{x}, \bm{y}) = (\bm{a}, \bm{p}, \bm{m}).
     $
     Then $\varphi$ is a bijection between the desired sets by 1 through 4.
\end{proof}
\begin{lemma}
\label{midpoint is heavy}
    Given integer vectors $\bm{k}_1 \le \bm{k}_2$ and $\bm{b}$ with proper dimensions, define $\bm{k}:= \bm{k}_2 + \bm{k}_1$. If $M\bm{k}$ is even, then we have: 
    \[
    B_{(\bm{k}_1, \bm{k}_2)}(\bm{b}) \le B_{\bm{(\bm{k}_1, \bm{k}_2)}}\left(\frac{M \bm{k}}{2}\right).
    \]
\end{lemma}
\begin{proof}
    Proof is done by induction over $|\bm{k}_2 - \bm{k}_1|$. Suppose $|\bm{k}_2 - \bm{k}_1| = 0$, i.e. $\bm{k}_2 = \bm{k}_1$. Then $B_{(\bm{k}_1, \bm{k}_2)}(\frac{M\bm{k}}{2}) = 1$, and $B_{(\bm{k}_1, \bm{k}_2)}(\bm{b})$ is either zero or one, depending on whether $M \bm{k}_1 = \bm{b}$ or not. Thus the base case holds. 

    By \Cref{B^2 in terms of E and B}, write: 
    \[
    B_{(\bm{k}_1,\bm{k}_2)}(\bm{b})^2 = \sum_{\substack{\bm a \in \Z^n\\\bm{0} \le \bm{a} \le \bm{k}_2- \bm{k}_1\\ E(\bm{a}) > 0}} E(\bm{a}) B_{(\bm{k}_1, \bm{k}_2 - \bm{a})}\left(\bm{b} - \frac{M\bm{a}}{2}\right),
    \]
    and:
    \[
    B_{(\bm{k}_1,\bm{k}_2)}\left(\frac{M\bm{k}}{2}\right)^2 = \sum_{\substack{\bm a \in \Z^n\\\bm{0} \le \bm{a} \le \bm{k}_2- \bm{k}_1\\ E(\bm{a}) > 0}} E(\bm{a}) B_{(\bm{k}_1, \bm{k}_2 - \bm{a})}\left( \frac{M(\bm{k - a})}{2}\right).
    \]
    The term $\bm{a} = \bm{0}$ contributes $B_{(\bm{k}_1,\bm{k}_2)}(\bm{b})$ and $B_{(\bm{k}_1,\bm{k}_2)}\left(\frac{M\bm{k}}{2}\right)$ to the right hand sides of the first and the second identity respectively. By moving these terms to the left hand side, we get: 
    \[
    B_{(\bm{k}_1,\bm{k}_2)}(\bm{b})^2 - B_{(\bm{k}_1,\bm{k}_2)}(\bm{b}) = \sum_{\substack{\bm a \in \Z^n,\, \bm a \neq \bm 0\\\bm{0} \le \bm{a} \le \bm{k}_2- \bm{k}_1\\ E(\bm{a}) > 0}} E(\bm{a}) B_{(\bm{k}_1, \bm{k}_2 - \bm{a})}\left(\bm{b} - \frac{M\bm{a}}{2}\right),
    \]
    and: 
    \[ B_{(\bm{k}_1,\bm{k}_2)}\left(\frac{M\bm{k}}{2}\right)^2 -B_{(\bm{k}_1,\bm{k}_2)}\left(\frac{M\bm{k}}{2}\right) = \sum_{\substack{\bm a \in \Z^n,\, \bm a \neq \bm 0\\\bm{0} \le \bm{a} \le \bm{k}_2- \bm{k}_1\\ E(\bm{a}) > 0}} E(\bm{a}) B_{(\bm{k}_1, \bm{k}_2 - \bm{a})}\left( \frac{M(\bm{k - a})}{2}\right).
    \]
    By induction hypothesis, for any $\bm{0} \le \bm{a} \le \bm{k}_2 - \bm{k}_1, \bm{a} \neq \bm{0}$:
    \[
    B_{(\bm{k}_1, \bm{k}_2 - \bm{a}) }\left( \bm{b} - \frac{M\bm{a}}{2}\right) \le B_{(\bm{k}_1 , \bm{k}_2 - \bm{a})} \left( \frac{M(\bm{k} - \bm{a}) }{2}\right),
    \]
    implying that:
    \[
    B_{(\bm{k}_1,\bm{k}_2)}\left(\frac{M\bm{k}}{2}\right)^2 -B_{(\bm{k}_1,\bm{k}_2)}\left(\frac{M\bm{k}}{2}\right) \ge B_{(\bm{k}_1,\bm{k}_2)}(\bm{b})^2 - B_{(\bm{k}_1,\bm{k}_2)}(\bm{b}).
    \]
    This conclusion completes the proof except for the case that $B_{(\bm{k}_1,\bm{k}_2)}\left(\frac{M\bm{k}}{2}\right) = 0$ and  $B_{(\bm{k}_1,\bm{k}_2)}(\bm{b}) = 1$. However, $\mathcal{B}_{(\bm{k}_1,\bm{k}_2)}\left(\frac{M\bm{k}}{2}\right)$ is a nonempty polytope since $\frac{\bm{k}}{2} \in \mathcal{B}_{(\bm{k}_1, \bm{k}_2)} (\frac{M\bm{k}}{2})$, and therefore has at least one vertex which is integral by \Cref{unimodular polytopes have integer vertices}. Thus $B_{(\bm{k}_1,\bm{k}_2)}(\frac{M\bm{k}}{2}) \ge 1$ for any $\bm{k}$, fixing the issue.                  
\end{proof}

\begin{lemma}
\label{unimodular polytopes have integer vertices}
    Suppose $M \in \Z^{m \times n}$ is a unimodular matrix and $\bm b \in \Z^m, \bm k_1, \bm k_2 \in \Z^n$. If the polytope $\mathcal{B}_{(\bm k_1, \bm k_2)}(\bm b)= \{ \bm{y}: M \bm y = \bm b, \bm k_1 \le \bm y \le \bm k_2\}$ is nonempty, then all of its vertices are integral. 
\end{lemma}
\begin{proof}
    Let $r = \rank(M)$ and fix $\bm{y}^*$ to be a vertex of $\mathcal{B}_{(\bm k_1, \bm k_2)}(M, \bm b)$. We say that a constraint $k_{1,i} \le y_i$ or $y_i \le k_{2,i}$ is active at $\bm{y}^*$ if it holds with equality for $\bm{y}^*$. The equality constraints $M\bm{y}^* = \bm{b}$ are automatically active. 
    
    Without loss of generality, assume that the first $r$ rows of $M$ are linearly independent, and let $M_{[r]}$ be the submatrix of $M$ containing these $r$ rows, and let $\bm b_{[r]}$ be the vector $\bm{b}$ restricted to its first $r$ entries.
    
    Note that the normal vectors of the active constraints of $\mathcal{B}_{\bm k_1, \bm k_2}(\bm b)$ at $\bm{y}^*$ must span $\R^n$. Otherwise there will exist some $\bm{v} \in \R^n$ orthogonal to all of the active constraints. Therefore, any point $\bm{y}^* \pm \epsilon \bm{v}$ belongs to $P$ for small enough $\epsilon \in \R_{\ge 0}$, contradicting the assumption that $\bm{y}^*$ is a vertex of $M$. 

    So in addition to the $r$ constraints of $M_{[r]}$ that are automatically active at $\bm{y}^*$, there are (at least) $n-r$ constraints of $\bm{k}_1 \le \bm y \le \bm k_2$ that are active at $\bm{y}^*$ and are linearly independent together with the equality constraints of $M_{[r]}$. After permuting the columns of $M$, assume that $y_i = d_i$ for $i\in [n-r]$ where $d_i \in \{k_{1,i}, k_{2,i} \}$.  For simplicity, let $A$ be the submatrix of the first $n-r$ columns of $M_{[r]}$ and $B$ its last $r$ columns, so that $M_{[r]} = \begin{bmatrix}
        A & B
    \end{bmatrix}$. We know that $\bm{y}^*$ satisfies: 
    \[
    \begin{bmatrix}
       A & B \\ \mathrm{id}_{(n-r) \times (n-r)} & 0_{(n-r) \times r} 
    \end{bmatrix} \bm{y}^* = \begin{bmatrix}
        \bm b_{[r]} \\ \bm d
    \end{bmatrix}, \det \begin{bmatrix}
       A & B \\ \mathrm{id}_{(n-r) \times (n-r)} & 0_{(n-r) \times r} 
    \end{bmatrix}\neq 0,
    \]
  In particular, we must have $\det{B} \neq 0$. Since $B$ is a non-singular $r \times r$ submatrix of $M$, its determinant is either $+1$ or $-1$ and $B^{-1}$ is an integer matrix. Thus:
  \[
  {y}^*_j = \left[ B^{-1}(\bm{b}_{[r]} - A \bm{d}) \right]_j \in \Z, \quad \forall j \in \{ n-r + 1, \cdots, n\}
  \]
  and $\bm{y}^*$ is an integer point. 
\end{proof}

We now have all the tools to prove \Cref{thm- Midpoint log-concavity for TU matrices}. 

\begin{proof}[{Proof of \Cref{thm- Midpoint log-concavity for TU matrices}}]
    Expand $T(\bm{b}, \bm{k}_1, \bm{k}_2; \bm{w})^2 - T(\bm{b^+}, \bm{k}_1, \bm{k}_2; \bm{w}) \cdot T(\bm{b^-}, \bm{k}_1, \bm{k}_2; \bm{w})$:
    \begin{align*}
        T(\bm{b}, \bm{k}_1, \bm{k}_2; \bm{w})^2 &= \sum_{2 \bm{k}_1 \le \bm{\alpha} \le 2 \bm{k}_2} \left( \sum_{\substack{\bm{k}_1 \le \bm{x,y} \le \bm{k}_2 \\ M\bm{x} = M \bm{y} = \bm{b}\\ \bm{x}+ \bm{y} = \bm{\alpha}}}1 \right) \bm{w^{\bm{\alpha}}}\\&= \sum_{\substack{2\bm{k}_1 \le \bm{\alpha} \le 2\bm{k}_2 \\ M\bm{\alpha} = 2 \bm{b}}} \left( \sum_{\substack{\bm{k}_1 \le \bm{x} \le \bm{k}_2 \\ \bm{k_1} \le \bm{\alpha} - \bm{x} \le \bm{k}_2 \\ M \bm{x} = \bm{b}}}  1 \right) \bm{w^{\alpha}}\\& = \sum_{\substack{2\bm{k}_1 \le \bm{\alpha} \le 2\bm{k}_2 \\ M\bm{\alpha} = 2 \bm{b}}} B_{(\max \{ \bm{k}_1, \bm{\alpha} - \bm{k}_2\} , \min\{ \bm{k}_2, \bm{\alpha} - \bm{k}_1\})}(\bm{b}) \bm{w^{\alpha}}
    ,
    \end{align*}
    and: 
    \begin{align*}
    T(\bm{b^+}, \bm{k}_1, \bm{k}_2; \bm{w}) \cdot T(\bm{b^-}, \bm{k}_1, \bm{k}_2; \bm{w}) &= \sum_{2 \bm{k}_1 \le \bm{\alpha} \le 2 \bm{k}_2 } \left( \sum_{\substack{\bm{k}_1 \le \bm{x,y} \le \bm{k}_2\\ M\bm{x} = b^+, M\bm{y} = b^-\\ \bm{x} + \bm{y} = \bm{\alpha}}}1 \right) \bm{w^\alpha} \\&= \sum_{\substack{ 2 \bm{k}_1 \le \bm{\alpha} \le 2 \bm{k}_2 \\ M \bm{\alpha} = 2 \bm{b}}} B_{(\max \{ \bm{k}_1, \bm{\alpha} - \bm{k}_2\} , \min\{ \bm{k}_2, \bm{\alpha} - \bm{k}_1\})}(\bm{b}^+) \bm{w^{\alpha}}. 
    \end{align*}
    where the $\min, \max$ are taken coordinate-wise. 
    For each $i$, if $\alpha_i \le k_{1,i} + k_{2,i}$ then the $i$-th coordinate of $\max \{ \bm{k}_1, \bm{\alpha} - \bm{k}_2\} $ is $k_{1,i}$ and the $i$-th coordinate of $\min\{ \bm{k}_2, \bm{\alpha} - \bm{k}_1\}$ is $\alpha_i - k_{1,i}$. If $\alpha_i \ge k_{1,i} + k_{2,i}$ then the $i$-th coordinate of $\max \{ \bm{k}_1, \bm{\alpha} - \bm{k}_2\} $ is $\alpha_i - k_{2,i}$ and the $i$-th coordinate of $\min\{ \bm{k}_2, \bm{\alpha} - \bm{k}_1\}$ is $k_{2,i}$. Thus $\max \{ \bm{k}_1, \bm{\alpha} - \bm{k}_2\} + \min\{ \bm{k}_2, \bm{\alpha} - \bm{k}_1\} = \bm{\alpha}$.
    
    it follows from \Cref{midpoint is heavy} that all the coefficients of $T(\bm{b}, \bm{k}_1, \bm{k}_2; \bm{w})^2 - T(\bm{b^+}, \bm{k}_1, \bm{k}_2; \bm{w}) \cdot T(\bm{b^-}, \bm{k}_1, \bm{k}_2; \bm{w})$ are nonnegative. 
\end{proof}

\begin{corollary}
\label{Ehrhart polynomials of unimodular polytopes are log-concave}
    Let $M \in \Z^{m \times n}$ be a unimodular matrix and $\bm b \in \Z^m$. Suppose $P = \{\bm{y} \in \R_{\ge 0}^n: M \bm{y} = \bm{b}\}$ is a bounded polytope. Then $E_{P}(t)^2 \ge E_{P}(t-1)\cdot E_{P}(t+1)$ for all $ t \in \Z_{\ge 1}$. 
\end{corollary}
\begin{proof}
By \Cref{thm- Midpoint log-concavity for TU matrices}:
\[
T(t\bm{b}, \bm 0, \bm k; \bm 1)^2 \ge T((t-1)\bm{b}, \bm 0, \bm k; \bm 1) \cdot T\left((t+1)\bm{b}, \bm 0, \bm k; \bm 1\right)
\]
As $\bm{k} $ tends to $ \infty$ coordinate-wise we get: 
\[
\#\{\bm{y} \in \Z_{\ge 0}^n: M \bm{y} = t \bm b \}^2 \ge \#\{\bm{y} \in \Z_{\ge 0}^n: M \bm{y} = (t-1) \bm b \} \cdot \#\{\bm{y} \in \Z_{\ge 0}^n: M \bm{y} = (t+1) \bm b \}.
\]
\end{proof}

\begin{remark}
\label{why our methods dont extend to IDP polytopes}
    We remark that if $M$ is unimodular and $\bm b \in \Z^m, \bm k_1, \bm k_2 \in \Z^n$ are given, the polytope $\mathcal{B}_{(\bm k_1, \bm k_2)}(M, \bm b)$ is indeed IDP. However, not all IDP polytopes are obtained this way. Furthermore, \Cref{unimodular polytopes have integer vertices} is not generally correct for an arbitrary IDP polytope. In particular, with the added assumption that $M$ has full row rank, one can show that $M$ is unimodular iff $\mathcal{B}_{(\bm{k}_1, \bm{k}_2)}(\bm b)$ has integral vertices for any choice of $\bm b, \bm k_1, \bm k_2$. Thus, the proof of \Cref{thm- Midpoint log-concavity for TU matrices} does not extend to IDP polytopes. In fact, \Cref{Ehrhart polynomials of unimodular polytopes are log-concave} fails for general IDP polytopes, see \cite{Fer26}.

\end{remark}
\subsection{Solution Laurent polynomials of TU matrices}
\label{subsection: solution Laurent polynomials are LC}
   Suppose $M \in \{0, \pm 1 \}^{m \times n}$ is a totally unimodular matrix. For any $\bm{b} \in \Z^{m}$, we want to find a lower bound for $\#\{ \bm{y} \in \Z_{\ge 0}^n: M \bm{y} = \bm{b} \} $. We encode these quantities as coefficients of a Laurent polynomial and prove that this Laurent polynomial is $\LC$. We will then use \Cref{capacity bound for LC polynomials} to obtain lower bounds for its coefficients in \Cref{section: capacity bounds}.
   

Suppose $\bm a_1, \cdots, \bm a_n \in \Z^m$ are the columns of $M$. For any integer vector $\bm{k} \in \Z_{\ge 0} ^n$, define the $\bm{k}$-solution Laurent polynomial of $M$ to be:
\[
P_{M,\bm{k}}(x_i : i \in [m]) = \prod_{j = 1}^n \left(\sum_{\ell =0}^{k_{j}}  \bm x^{\ell\bm a_j}\right) = \sum_{\bm{b} \in \Z^m} B_{(\bm{0}, \bm{k})}(\bm{b}) \bm{x^{b}},
\]
where $B_{(\bm{k}_1,\bm{k}_2)}(\bm{b}) =\# \{\bm{y} \in \Z^n: M \bm{y} = \bm{b} , \bm{k}_1 \le \bm{y} \le \bm{k}_2 \}$, as defined in \Cref{subsection-midpoint log-concavity for tu}.

The goal of this section is to prove that $P_{M, \bm{k}}$ is $\LC$ for any TU matrix $M$ and $\bm{k} \in \Z^n$. See \Cref{why pAB is not LC for unimodular matrices} on why we cannot prove a similar statement for unimodular matrices.

Let us first derive a simpler equivalent condition for $P_{M, \bm k}$ to be $\LC$. For the sake of simplicity, if $M$ is an $m \times n$ matrix and $S \subseteq [m]$, we denote by $M_S$ the submatrix of $M$ consisting of the rows indexed by $S$.
\begin{proposition} 
\label{when is PM in LC}
    Suppose $M$ is a TU matrix. Then $P_{M , \bm{k}}$ is $\LC$ if and only if the following condition holds: 
    
    Choose an arbitrary subset $S \subseteq [m]$ of the rows of $M$. Let $B = M_S$ and $A = M_{[m] \setminus S}$. Further choose an arbitrary $\bm{b} \in \Z^{|S|}$. Then any positive evaluation of the Laurent polynomial $\sum_{\substack{\bm{0} \le \bm{y} \le \bm{k} \\ B\bm{y} = b  }} \bm{x}^{A\bm{y}}$ in all but one variable has log-concave coefficients. 
\end{proposition}
\begin{proof}
By \Cref{definition: LC polynomials}, $P_{M,\bm{k}}$ is $\LC$ if for any $S \subseteq [m]$ and $\bm{b} \in \Z^{|S|}$, any positive evaluation of $[\bm{x}_S^{\bm{b}}]p_{M, \bm{k}}$ in all but one variable has log-concave coefficients. We can write: 
\begin{align*}
    [\bm{x}_S^{\bm{b}}]p_{M,\bm{k}}(\bm{x}) &= [\bm{x}_S^{\bm{b}}]\sum_{\bm{0} \le \bm{y} \le \bm{k}} \bm{x}^{M\bm{y}}\\
    &= \sum_{\substack{\bm{0} \le \bm{y} \le \bm{k}\\ B\bm{y} = \bm{b}}} \bm{x}^{A\bm{y}},
\end{align*}
    completing the proof. 
\end{proof}
\Cref{when is PM in LC} already reveals the connection between $\LC$ properties of $\bm{k}$-solution Laurent polynomials of TU matrices and midpoint log-concavity for TU polytopes. The following proposition completes this connection. 
\begin{proposition}
    \label{pAB is in LC}
    Suppose $\begin{bmatrix}
        A\\B
    \end{bmatrix}$ is a TU matrix and assume $A$ is an $m \times n$ matrix. Then any positive evaluation of $\sum_{\substack{\bm{0} \le \bm{y} \le \bm{k} \\ B\bm{y} = b  }} \bm{x}^{A\bm{y}}$ in all but one variable has log-concave coefficients. 
\end{proposition}
\begin{proof}
Without loss of generality, evaluate $\sum_{\substack{\bm{0} \le \bm{y} \le \bm{k} \\ B\bm{y} = b  }} \bm{x}^{A\bm{y}}$ in all but the last variable, and let $q(x)$ be the resulting Laurent polynomial. 
   It remains to check that $\{ [x^{\beta}]q \}_{\beta \in \Z}$ is a log-concave sequence.

    Let $\bm{a}_{(1)}, \bm{a}_{(2)}, \cdots, \bm{a}_{(m)}$ be the rows of $A$, and suppose we are evaluating $p$ at $y_1 = r_1, \cdots, y_{m-1} = r_{m-1}$. Then:
    \begin{align*}
        q(x) &= \sum_{\substack{ \bm{0} \le \bm{y} \le \bm{k}\\ B \bm{y} = \bm{b}}} r_1^{\bm{a}_{(1)}^T \bm{y}}\cdots r_{m-1}^{\bm{a}_{(m-1)}^T \bm{y}}x^{\bm{a}_{(m)}^T \bm{y}}\\
        &= \sum_{\substack{\bm{0} \le \bm{y} \le \bm{k}\\ B \bm{y} = b \\ A \bm{y} = \bm{w}}}r_1^{w_1} \cdots r_{m-1}^{w_{m-1}}x^{w_m}
    \end{align*}
Since the set $\{ \bm{y}: B \bm{y} = \bm{b}, \bm{0} \le \bm{y} \le \bm{k}\}$ is bounded, one can choose some integer vector $\bm{\gamma}$ such that $ - \bm{\gamma} \le \bm{w} = A \bm{y} \le \bm{\gamma}$ for any  $\bm{y}$ in this set. Then: 
\[
    q(x) = \sum_{\beta \in \Z} \left( \sum r_1^{w_1} \cdots r_{m-1}^{w_{m-1}}\right) x^{\beta},
\]
where the second sum is over the set: 
\[
T = \left\{ \begin{bmatrix}
    \bm{y} \\ \bm{w}
\end{bmatrix}: \begin{bmatrix}
    B & \bm{0} \\ {A} &- \mathrm{id} \\ 
\end{bmatrix}\begin{bmatrix}
    \bm{y} \\ \bm{w}
\end{bmatrix} = \begin{bmatrix}
    \bm{b} \\ \bm{0}
\end{bmatrix}, \quad w_m = \beta, \quad\begin{bmatrix}
    \bm{0} \\ -\bm{\gamma}
\end{bmatrix} \le \begin{bmatrix}
    \bm{y} \\ \bm{w}
\end{bmatrix} \le  \begin{bmatrix}
    \bm{k} \\ \bm{\gamma}
\end{bmatrix}\right\}
\]
Let $M$ be the matrix that is obtained by by appending a row of $e_{n+m}^T = \begin{bmatrix}
    0 & \cdots & 0 & 1
\end{bmatrix}$ row to the end of $\begin{bmatrix}
    B & {0} \\ A & -\mathrm{id} \\ 
\end{bmatrix}$, i.e. $M = \begin{bmatrix}
    B & 0 \\ A & -\mathrm{id} \\ \multicolumn{2}{c}{e_{n+m}^T} 
\end{bmatrix}$. By basic properties of TU matrices, $M$ is TU. Then we can rewrite:
\[
T =\left\{  \begin{bmatrix}
    \bm{y} \\ \bm{w}
\end{bmatrix} : M  \begin{bmatrix}
    \bm{y} \\ \bm{w}
\end{bmatrix} = \begin{bmatrix}
    \bm{b} \\ \bm{0} \\ \beta
\end{bmatrix},  \quad \begin{bmatrix}
    \bm{0} \\ -\bm{\gamma}
\end{bmatrix} \le \begin{bmatrix}
    \bm{y} \\ \bm{w}
\end{bmatrix} \le  \begin{bmatrix}
    \bm{k} \\ \bm{\gamma}
\end{bmatrix}\right\} = B_{\begin{bmatrix}
    \bm{0} \\ -\bm{\gamma}
\end{bmatrix}, \begin{bmatrix}
    \bm{k}\\ \bm{\gamma}
\end{bmatrix}}\left( M,\begin{bmatrix}
    \bm{b} \\ \bm{0} \\ \beta
\end{bmatrix} \right).\] 

Further rewrite $q(x)$ as follows: 
\[
q(x) = \sum_{\beta \in \Z} T\left(M,\begin{bmatrix}
    \bm{b} \\ \bm{0} \\ \beta
\end{bmatrix}, \begin{bmatrix}
    \bm{0} \\ -\bm{\gamma}
\end{bmatrix}, \begin{bmatrix}
    \bm{k}\\ \bm{\gamma}
\end{bmatrix}; \begin{bmatrix}
    1&  1& \cdots &1 & r_1 & \cdots& r_{m-1} & 1
\end{bmatrix}^T\right) x^{\beta}.
\]

So the coefficients of $q(x)$ form the following sequence: 
\[
\left \{ T\left(M,\begin{bmatrix}
    \bm{b} \\ \bm{0} \\ \beta
\end{bmatrix}, \begin{bmatrix}
    \bm{0} \\ -\bm{\gamma}
\end{bmatrix}, \begin{bmatrix}
    \bm{k}\\ \bm{\gamma}
\end{bmatrix}; \begin{bmatrix}
    1&  1& \cdots &1 & r_1 & \cdots& r_{m-1} & 1
\end{bmatrix}^T\right)\right\}_{\beta \in \Z}
\]

Observe that this sequence does not have any internal zeros. This is because given any integer vectors $\bm \beta_1, \bm \beta_2$, the set $\{ \bm{d}: M \bm y = \bm d \text{ for some } \bm{\beta_1} \le \bm{y} \le \bm \beta_2\}$ is convex, and if $\bm{d}$ is an integer vector in this set, then $\bm \beta_1 \le \bm{y} \le \bm \beta_2, M \bm y = \bm d$ has an integer solution by \Cref{unimodular polytopes have integer vertices}. 

Thus, this sequence is log-concave due to \Cref{thm- Midpoint log-concavity for TU matrices} and because $2\begin{bmatrix}
    \bm{b} \\ \bm{0} \\ \beta
\end{bmatrix} = \begin{bmatrix}
    \bm{b} \\ \bm{0} \\ \beta -1 
\end{bmatrix} + \begin{bmatrix}
    \bm{b} \\ \bm{0} \\ \beta +1
\end{bmatrix}$.
\end{proof}

\begin{remark}
\label{why pAB is not LC for unimodular matrices}
Suppose $M$ is a general (not necessarily totally) unimodular matrix. To prove that $P_{M, \bm k}$ is $\LC$, we would need to show that $\sum_{\substack{\bm{0} \le \bm{y} \le \bm{k} \\ B\bm{y} = b  }} \bm{x}^{A\bm{y}}$ is $\LC$ for every partition of the rows of $M$ into two submatrices $A$ and $B$. However, the proof of \Cref{pAB is in LC} does not always hold for unimodular matrices.
That said, any unimodular polytope is integrally equivalent to a TU polytope and the bounds we present in \Cref{section: capacity bounds} apply to unimodular polytopes after a suitable linear transformation.

\end{remark}

\begin{corollary}
\label{k-solution Laurent polynomial of tu matrices are lc}
    By \Cref{when is PM in LC} and \Cref{pAB is in LC}, the $\bm{k}$-solution Laurent polynomial of any TU matrix is $\LC$ for any $\bm{k} \ge \bm{0}$.  
\end{corollary}

\section{Capacity bounds for lattice points of TU polytopes}
\label{section: capacity bounds}
With \Cref{k-solution Laurent polynomial of tu matrices are lc} at our disposal, we can use \Cref{capacity bound for LC polynomials} to derive a capacity lower bounds for $B_{(\bm{0}, \bm{k})}(M,\bm{b}) = \# \{ \bm y \in \Z^n: M \bm y = \bm b, \bm 0 \le \bm y \le \bm k\}$. In fact, \Cref{capacity bound for LC polynomials} gives us the following bound: 
\[
B_{\bm 0 , \bm k}({M, \bm b})\ge \cpc_{\bm b}(P_{M, \bm k}) \cdot \prod_{i =1}^m \frac{|\beta_i(\bm k)- b_i|^{|\beta_i(\bm k)- b_i|}}{(1+|\beta_i(\bm k)- b_i|)^{1 + |\beta_i(\bm k) - b_i|}}, \tag{1} \label{bound 1}
\]
where each $\beta_i(\bm k)$ is a bound for $\Deg_{i}^{\bm b}(P_{M, \bm{k}})$. We would then like to send $\bm k$ to $\infty$, in the sense that each $k_i$ tend to $\infty$. It is easy to see that the left hand side of this inequality converges to $B_M(\bm b):= \{\bm y \in \Z_{\ge 0}^n: M \bm y = \bm b \}$ as $\bm k \to \infty$. We will show in \Cref{cap of PMk converges to cap PM} that if $M$ has finite fibers, then $\lim_{\bm k \to \infty} \cpc_{\bm b} P_{M, \bm k} = \cpc_{\bm b} P_{M}$, where $P_M (\bm x) := \sum_{\bm y \ge 0} \bm x^{M \bm y}$. 

Unfortunately, the product of constants in bound (\ref{bound 1}) might tend to $0$ as $\bm k \to \infty$. Let $C$ be the convex hull of $\supp(P_{M, \bm k})$. Then $\beta_i(\bm k)$ is a bound for the $i$-th coordinate of vectors in $C \cap \{ \bm z: z_{i+1} = b_{i+1}, \cdots, z_m = b_m \}$. Let $\bm r$ be the $i$-th row of $M$. We can naively bound the $i$-th coordinate of any $\bm z \in C$ by $\bm k^T (\bm r - (\bm r)_+) \le z_i \le \bm k^T(\bm r)_+$ since $\bm z = M \bm y$ for some $\bm 0 \le \bm y \le \bm k$. However, if $M$ is a general TU matrix, $C$ can be arbitrary enough that fixing the values of $z_{i+1} , \cdots, z_m$ for a point $\bm z \in C$ doesn't restrict $z_i$ beyond the trivial bound we obtained from $\bm r$. Using either of these trivial bounds for $\beta_i(\bm k)$ is an issue as they make the constant in bound \ref{bound 1} converge to $0$ as $\bm k \to \infty$ (note that $\lim_{t \to \infty} {t^t}/{(1+t)^{(1+t)}} = 0$). To fix this problem, we utilize the geometry of the cone generated by $C$. 

In \Cref{subsection: geometry of the cone}, we study the geometry of $K$, the cone generated by $\supp(P_{M, \bm k})$. We show that we can transform $K$ using a linear transformation $U$ such that for any $\bm z \in UK$, the values of $z_{i+1}, \cdots, z_m$ give us a bound $\beta_i$ for $z_i$ that is independent from $\bm k$. We use this construction in \Cref{subsetction: new approach-general capacity bounds for tu polytopes} to prove capacity lower bounds for $B_{M}(\bm b)$.

\subsection{Geometry of $ \cone (\supp P_{M, k})$}
\label{subsection: geometry of the cone}
Let us recall some basic terminology and facts about cones and flags of faces of cones. 

A cone $K \subseteq \R^n$ is a set closed under addition and multiplication by positive numbers. In the context of this paper, the cone generated by vectors $\bm v_1, \cdots, \bm v_k \in \R^n$ denoted by $\cone\{\bm v_1 , \cdots, \bm v_k \}$ is the set $\{ \sum_{i = 1}^k \lambda_i \bm v_i: \lambda_i \ge 0  \quad \forall i \in [k] \}$. We say that a cone is pointed if whenever $\bm v, -\bm v \in K $, we have $\bm v  =  \bm 0$. Furthermore, the dimension of a cone is defined to be the dimension of its affine hull. In the case that $K$ is generated by $\bm v_1 , \cdots, \bm v_k$, its dimension is equal to the dimension of $\sspan \{\bm v_1, \cdots, \bm v_k \}$. If $ S $ is an infinite set, we define the cone generated by $S$ to be $\cone(S):= \{ \sum_{i = 1} ^n  \lambda_i s_i: n \in \N, \lambda_i \ge 0, s_i \in S \quad \forall i \in [n] \}$. 

A face of a cone $K$ is the intersection of $K$ with one of its supporting hyperplanes $H$. In other words, $F \subseteq K$ is said to be a face of $K$ if there exists some $\bm a \in \R^n$ such that:
\[
\bm a^T \bm y > 0 \quad \forall \bm y \in K \setminus F, \qquad \bm a^T \bm y = 0 \quad \forall \bm y \in F.
\]
We say that $F$ is a facet of $K$ if it is a maximal proper face. It is a well-known fact that a face is the intersection of all the facets containing it. Note that a face of a cone remains a cone since it is closed under addition and multiplication with positive numbers. Moreover, if $K = \cone \{ \bm v_1, \cdots, \bm v_k \}$, then each face $F$ of $K$ is generated as a cone by a subset of $\{ \bm v_1, \cdots, \bm v_k \}$. In particular, $F = \cone \{ \bm v_i : i \in [n], \, \bm a^T \bm v_i = 0 \}$, where $\bm a \in \R^n$ defines the supporting hyperplane corresponding to $F$. 

Lastly, if $K = \cone \{ \bm v_1 , \cdots, \bm v_k\}$ is an $m$-dimensional pointed cone, then it is well known that one can construct a flag of faces of $K$:
\[
\{ \bm 0 \} = F_0 \subset F_1 \subset \cdots \subset F_{m-1} \subset F_m = K, \qquad \dim(F_i) = i \quad \forall i \in [m].
\]

Fix a full row rank TU matrix $M \in \Z^{m \times n}$ with columns $\bm a_1 , \cdots, \bm a_n \in \Z^m$, and $\bm b \in \Z^m$. It is easy to see that $\cone ( \supp P_{M, \bm k}) = \cone \{ \bm a_1, \cdots, \bm a_n \}$. We show that if the fibers of $M$ are finite, then the cone generated by $\bm a_1, \cdots, \bm a_n$ is a pointed cone. Thus, we can construct a flag of faces for $\cone \{ \bm a_1, \cdots, \bm a_m \}$. We use this flag of faces to find a linear transformation $U: \R^m \to \R^m$, We further show that for any facet of $\cone \{ \bm a_1, \cdots, \bm a_m \}$, there exists some supporting hyperplane corresponding to it that has a normal vector in $\{ 0, \pm 1\}^m$. We use this latter fact to show that $\Deg_{i}^{U \bm b}(P_{UM, \bm k})$ has a lower bound independent from the choice of $\bm k$. In fact, we prove that the following statement holds: 
\begin{proposition}
\label{a good linear transformation for supp PM}
    Suppose $M \in \Z^{m \times n}$ is a full row rank TU matrix and $\bm b \in \Z^m$ and $\bm k \in \Z_{\ge 1}^n$. Assume further that the fiber $\{  \bm y: M \bm y = \bm b, \bm y \ge \bm 0 \}$ is finite.  Then there exists some TU matrix $U \in \Z^{m \times m}$ such that: 
    \begin{enumerate}
        \item $N:= UM$ remains TU, 
        \item 0 is a lower bound for the set $\Deg_{m}(P_{N, \bm k})$ and $(U \bm b)_i - \|\bm b \|_1$ a lower bound for the set $\Deg^{U \bm b}_{i}(P_{N, \bm k})$ for any $i \in [m-1]$.         
        \item $[ \bm x^b ] P_{M, \bm k} = [\bm x^{U \bm b}] P_{N, \bm k}$,
        \item $\cpc_{\bm b}(P_{M, \bm k}) = \cpc_{U \bm b} (P_{N, \bm k})$, 
        
    \end{enumerate}
    \end{proposition}

Let us analyze the geometry of $\cone (\supp P_{M, \bm k})$ before proving \Cref{a good linear transformation for supp PM}. 
\begin{lemma}
\label{finite fibers criteria}
    Suppose $M \in \Z^{m \times n}$ is a TU matrix and $\bm b \in \Z^m$ is arbitrary such that $\{ \bm y \in \Z_{\ge 0}^n: M \bm y = \bm b \}$ is nonempty. Then $\{ \bm y \in \Z_{\ge 0}^n: M \bm y = \bm b\}$ is finite if and only if $\{ \bm y \in \Z_{\ge 0}^n: M \bm y = \bm 0\}= \{ \bm 0 \}$. 
\end{lemma}
\begin{proof}
   First, let us prove the forward direction of the statement. Suppose there exist some $\bm z \in \Z_{\ge 0}^n \setminus \{ \bm 0 \} $ such that $M \bm z = \bm 0$. Let $\bm y_0 \in \Z_{\ge 0}^n$ satisfy $M \bm y_0 = \bm b$. Then $\{ \bm y_0 + r \bm z: r \ge 0 \}$ is an infinite subset of $\{ \bm y \in \Z_{\ge 0 }^n : M \bm y = \bm b \}$, finishing this direction of the proof. 

   Now let us prove the other direction. Suppose $\{ \bm y \in \Z_{\ge 0}^n : M \bm y  = \bm b \}$ is an infinite set. Then by Dickson's lemma \cite[Lemma A, p. 414]{Dic13}, there exists two distinct vectors $\bm z_1 \ge \bm z_2$ in this set, and therefore $\bm z_1 - \bm z_2 $ is a non-zero element of the set $ \{ \bm y \in \Z_{\ge 0}^n: M \bm y = \bm 0 \}$.
\end{proof}

\Cref{finite fibers criteria} immediately implies that $\ker(M) \cap \Z_{\ge 0}^{n} = \{\bm 0\}$ iff $\{ \bm y \in \Z_{\ge 0}^n: M \bm y = \bm b\}$ is finite for any $\bm b \in \Z^m$. Going forward, we will say that such an $M$ has finite fibers. 

\begin{corollary}
\label{support forms a pointed cone}
    Suppose $M = [\bm a_1\mid \cdots \mid \bm a_n]$ is a full row rank TU matrix. If $\{\bm y \in \Z_{\ge 0}^n: M \bm y = \bm b \}$ is finite, then $\cone (\supp P_{M, \bm k}) = \cone \{ \bm a_1, \cdots, \bm a_n\}$ is a pointed cone for any $\bm k \in \Z_{\ge 1}^n$. 
\end{corollary}
\begin{proof}
    It is easy to see that the cone generated by $\supp(P_{M, \bm k})$ is the cone generated by the columns of $M$. Moreover, if $\bm v, - \bm v \in \cone \{\bm v_1, \cdots, \bm v_n\} $, then there exists some $\bm y_1, \bm y_2 \in \R_{\ge 0}^n$ such that:
    \[
    \bm v = M \bm y_1, \qquad - \bm v = M \bm y_2,
    \]
    therefore, $M (\bm y_1 + \bm y_2) = \bm 0$. By \Cref{finite fibers criteria} this implies that $\bm y_1 + \bm y_2 = \bm 0$. Since $\bm y_1, \bm y_2 \in \R_{\ge 0}^n$, we get $\bm y_1 = \bm y_2 = \bm 0$ and thus, $\bm v = - \bm v = \bm 0$.
\end{proof}

Next, we will show that for any facet of cone generated by $\supp(P_{M, \bm k})$, there exists a supporting hyperplane of $\cone (\supp P_{M, \bm k})$ that has a normal vector in $\{ 0, \pm 1\}^m$.

\begin{lemma}
\label{facets have small normal vectors}
    Suppose $M = [\bm a_1 \mid \cdots \mid \bm a_n] \in \Z^{m \times n}$ is a full row rank TU matrix. Let $F$ be a facet of $\cone  \{ \bm a_1, \cdots, \bm a_n \}$. Then there exists some $\bm u \in \{ 0 , \pm 1\}$ such that: 
    \[
    \bm u^T \bm y > 0 \quad \forall \bm y \in \cone \{ \bm a_1 , \cdots, \bm a_n \} \setminus F, \qquad \bm u^T \bm y = 0 \quad \forall \bm y \in F. 
    \]
    Moreover, if $\bm a$ is a column of $M$ satisfying $\dim \cone(F \cup \{ \bm a \}) = m$, we can construct $\bm u$ such that $\bm u^T \bm a = 1$. 
\end{lemma}

    \begin{proof}
        Define $C = \cone \{ \bm a_1, \cdots, \bm a_n\}$. Since $M = [ \bm a_1 \mid \cdots \mid \bm a_n]$ is full row rank, $C$ is an $m$-dimensional cone, and therefore $F$ is an $(m-1)$-dimensional cone generated by a subset of columns of $M$. Assume without loss of generality that $F = \cone \{ \bm a_1 , \cdots, \bm a_k \}$ and $\bm a_1, \ldots \bm a_{m-1}$ are linearly independent.
        
        Since $M = [\bm a_1 \mid \cdots \mid \bm a_n] $ is full row rank, we can choose another column $\bm a$ of $M$ such that $ \bm a_1 ,\cdots , \bm a_{m-1}, \bm a$ are linearly independent. Letting $A := [\bm a_1 \mid \cdots \mid  \bm a_{m-1} \mid \bm a]$, we have that $A$ is an invertible submatrix of $M$. Since $M$ is TU, we thus have that $A^{-1}$ is TU as well. Let $\bm u \in \{ 0, \pm 1\}^m$ be the last row of $A^{-1}$. Since $A^{-1} A = \text{id}$, we have 
        \[
        \bm u^T \bm a_i = \bm 0 \quad \forall i \in [m-1], \qquad \bm u^T \bm a = 1 > 0. 
        \]
        Since $\cone \{ \bm a_1, \cdots, \bm a_{m-1} , \bm a \}$ is a full-dimensional subcone of $C$ and $F$ spans a supporting hyperplane of $C$, we can write any $\bm y \in C$ as $\bm{y} = \lambda \bm a + \lambda_1 \bm a_1 + \cdots + \lambda_{m-1} \bm a_{m-1} $ for some $\lambda_1,\ldots,\lambda_{m-1}  \in \R$ and $\lambda \ge 0$. Thus
        \[
        \bm u^T \bm y = \lambda \bm u^T \bm a + \sum_{i = 1}^{m-1 } \lambda_i \bm u^T \bm a_i = \lambda,
        \]
        and $\bm u^T \bm y = 0 $ iff $\lambda = 0$. That is, $\bm u^T \bm y = 0$ whenever $\bm y \in F$ and $\bm u^T \bm y > 0$ whenever $\bm y \in  \setminus F$.  
    \end{proof}

    Let us prove one last lemma before proving \Cref{a good linear transformation for supp PM}.
\begin{lemma}
\label{product of two special TU matrices is TU}
    Suppose $M = \begin{bmatrix}
        A  \mid B
    \end{bmatrix}$
    is an $m \times n$ TU matrix, where $A$ is $m \times m$ and invertible. Then $A^{-1}B$ is also TU. 
\end{lemma}
\begin{proof}
    {Each $m \times m$ minor of $A^{-1}M$ is in $\{-1,0,+1\}$ since $M$ is TU and $\det(A^{-1}) \in \{-1,+1\}$. Since $A^{-1}M$ contains an identity block, it is thus TU itself. The result follows since $A^{-1}B$ is a submatrix.}

\end{proof}

We have the tools to prove \Cref{a good linear transformation for supp PM} now. 

\begin{proof} [ Proof of \Cref{a good linear transformation for supp PM}]
    We construct $U$ to be the inverse of a specific square submatrix of $M$. The construction of this square submatrix is as follows.

    Let $\bm a_1, \cdots, \bm a_n$ be the columns of $M$, and define $K:= \cone \{ \bm a_1, \cdots, \bm a_n \}$ be the cone generated by $\supp (P_{M , \bm k})$. By \Cref{support forms a pointed cone}, $K$ is a pointed cone. Construct a flag of faces for $K$:
    \[
    \{ \bm 0 \} = F_0 \subset F_1 \subset \cdots \subset F_{m-1} \subset F_m = K, 
    \qquad \dim F_i = i \quad \forall i \in [m].
    \]
    We know that each $F_i$ is generated as a cone by a subset of columns of $M$, say $ F_i = \cone \{ \bm a_j: j \in S_i \}$ for some $S_i \subseteq [n]$. Note that $\dim F_{i - 1} = \dim \sspan \{\bm a_j: j \in S_{i- 1} \} < \dim \sspan \{ \bm a_j: j \in S_i \} = \dim F_i$, and therefore, there exists some column of $M$ that belongs to $ \sspan \{\bm a_j: j \in S_{i+1} \} \setminus \sspan \{ \bm a_j: j \in S_i\}$.  Without loss of generality, let this column be $\bm a_i$. Note that $\bm a_i \in F_{i} \setminus F_{i-1}$, and that $\bm a_1, \cdots, \bm a_m$ are linearly independent vertices, 
    
    Let $A = M_{..,[m]}$. $A$ is an invertible TU matrix and so is its inverse $U := A^{-1}$. We will show that $U$ satisfies properties 1-4 stated in \Cref{a good linear transformation for supp PM}. 

    The first property follows from \Cref{product of two special TU matrices is TU} since $N:= UM$ is obtained by multiplying $ A^{-1}$ with $ [A\mid M_{.. , [m+1,n]}]$. 

    Let us prove that the second property holds as well. To show that $0$ is a lower bound for $\Deg_{m}(P_{N, \bm k}) = \{ (N \bm y)_m: \bm 0 \le \bm y \le \bm k, \bm y \in \Z^n\}$, it suffices to show that the last row of $N$ is a non-negative vector. Suppose $\bm u$ define a supporting hyperplane corresponding to the facet $F_{m-1}$. Recall that $\bm a_1, \cdots, \bm a_{m-1} \in F_{m-1}$ and $\bm a_m \not \in F_{m-1}$, so $\bm u^T \bm a_i = 0$ for $i \in [m-1]$ and $c:= \bm u^T \bm a_m > 0$. It is easy to check that $\bm u^T A = c \bm e_m^T$, therefore $\frac{1}{c} \bm u^T M = \bm e_m^T A^{-1} M $. Note that right hand side is just the last row of $N$, so we it remains to check that $\frac{1}{c}\bm u^T M \ge \bm 0$. Observe that $\bm u$ defines a supporting hyperplane for the cone generated by the columns of $M$, so $\bm u^T M \ge 0$ as any column of $M$ lies above this hyperplane. 
    
    To finish the proof of the second property, it remains to show that $(U \bm b)_i - \|\bm b \|_1$ is a lower bound for the set $\Deg_{i}^{U\bm b}(P_{N, \bm k})$ for any $i \in [m-1]$. Let $\bm \beta := U \bm b$. Fix $i \in [m-1]$, and let $H_i$ be a facet of $K$ containing $F_i$ such that $\bm a_{i +1} \not \in H_i$. We can choose such a facet because the intersection of all of the facets of $K$ containing $F_i$ is the face $F_i$ itself, and $\bm a_{i+1} \not \in F_i$.

    {Consider a supporting hyperplane for $K$ that corresponds to the facet $H_i$. Any vector spanned by  $H_i$ lies on this hyperplane, while $\bm a_{i+1} \in K \setminus H_i$ does not. So $\bm a_{i+1}  \not \in \sspan H_i$, and $\dim \cone (H \cup \{ \bm a_{i+1}\}) = 1+ \dim H  = m.$}
    By \Cref{facets have small normal vectors}, there exists some $\bm u_i \in \{0, \pm 1\}^m $ that defines the supporting hyperplane $H_i$, and is such that
    \[
    \bm u_i^T \bm a_j = 0 \quad \forall j \le i, \qquad \bm u_i^T \bm a_{i + 1} = 1, \qquad \bm u_i^T \bm y > 0 \quad \forall \bm y \in K \setminus H_i.
    \]
    
    Define $c_{i,j} := \bm u_i^T \bm a_j$ for any $j \in [n]$, and look at some $\bm z = N \bm y$ where $\bm 0 \le \bm y \le \bm k$ is arbitrary. We need to show that if $z_j = \beta_j$ for $j > i$, then $z_i$ is at least $(U \bm b)_i - \|\bm b \|_1$. We have: 
    \[
    \bm z = UM \bm y \implies A\bm z = M \bm y = \bm a_1 y_1 + \cdots + \bm a_n y_n,
    \]
    note that $A \bm z = M \bm y \in K$, so we should have $\bm u_{i-1}^T (A \bm z) \ge 0$. Write:
    \[
    \bm u_{i-1}^T(A \bm z) = \bm u_{i-1}^T \left( \bm a_1 z_1 + \cdots + \bm a_m z_m \right) = z_i + \sum_{j > i} c_{i-1,j}\beta_j \ge 0
    \]
    and $z_i \ge - \sum_{j > i} c_{i-1, j} \beta_j$. Let us do some more arithmetic with this lower bound. By construction of $\bm u_{i-1}$,  $c_{i-1,j} = \bm u_{i-1}^T \bm a_j = 0$ for $j < i$ and $c_{i-1,i} = \bm u_{i-1}^T \bm a_i = 1$. Hence we have the following identities: 
    \begin{align*}
        \beta_i - z_i \le \beta_i - ( - \sum_{j = i+1}^m c_{i-1, j} \beta_j) & = \sum_{j = 1}^m c_{i-1,j} \beta_j\\
        &= \sum_{j = 1}^m \bm u_{i-1}^T \bm a_j \beta_j \\
        &= \bm u_{i-1}^T \left( \begin{bmatrix}
            \bm a_1 \mid \bm a_2 \mid \cdots \mid \bm a_m
        \end{bmatrix} \bm \beta \right)\\
        &= \bm u_{i-1}^T AU \bm b\\
        & = \bm u_{i-1}^T \bm b \le \| \bm  b\|_1.
    \end{align*}
    The last inequality is due to the fact that $\bm u_{i-1} \in \{ 0, \pm 1\} ^m$. Thus $z_i \ge \beta_i - \| \bm b \|_1 $, proving the second property listed in \Cref{a good linear transformation for supp PM}.
    
Let us show that $U$ additionally satisfies properties 3 and 4. Since $U$ is invertible, we have that:
    \[
    \# \{ \bm y \in \Z^n: M \bm y = \bm b, \bm 0 \le \bm y \le \bm k\} = \# \{ \bm y \in \Z^n: UM \bm y = U \bm b, \bm 0 \le \bm y \le \bm k \},                    
    \]
    consequently, $[\bm x^{\bm b}] P_{M, \bm k} = [\bm x^{U \bm b}] P_{N, \bm k} $ and $U$ satisfies the third property as well.
    
    To prove the fourth property, write:  
    \begin{align*}
        \cpc_{\bm b} (P_{M, \bm k}) &= \inf_{\bm x > \bm 0} \frac{P_{M, \bm k}(\bm x)}{\bm x^{\bm b}} \\ &= \inf_{\bm x > \bm 0} \sum_{\bm \alpha \in \Z^m} [\bm x^{\bm \alpha}]P_{M, \bm k} \cdot \bm x^{\bm \alpha - \bm b}\\
        &= \inf_{\bm y \in \R^m} \sum_{\bm \alpha \in \Z^m} [\bm x^{U \bm \alpha}]P_{N, \bm k} \cdot e^{\bm y^T \cdot ( \bm \alpha - \bm b)},
    \end{align*}
    where $e^{\bm c} = (e^{c_1}, \cdots, e^{c_n})$ for any $\bm c \in \R^n$. Note that in the last step, we are using the third property for $N = UM$. Observe that as $U^T \in \R^{m \times m}$ is invertible, $\R^m = U^T \R^m$ and therefore:
    \begin{align*}
        \cpc_{\bm b}(P_{M, \bm k}) &= \inf_{\bm y \in \R^m} \sum_{\bm \alpha \in \Z^m} [\bm x^{U \bm \alpha}] P_{N, \bm k} \cdot e^{\bm (U^T \bm y)^T (\bm \alpha - \bm b) }\\
        &= \inf_{\bm y \in \R^m} \sum_{\bm \alpha \in \R^m} [\bm x^{U\bm \alpha}] P_{N, \bm k} \cdot e^{\bm y^T (U \bm \alpha - U \bm b)} \\
        &= \cpc_{U \bm b}(P_{N, \bm k}).
    \end{align*}    
    \end{proof}

\subsection{General capacity bounds for TU polytopes}
\label{subsetction: new approach-general capacity bounds for tu polytopes}
As proven in \Cref{k-solution Laurent polynomial of tu matrices are lc}, the $\bm{k}$-solution Laurent polynomial of any TU matrix is $\LC$. We can use \Cref{a good linear transformation for supp PM} alongside \Cref{capacity bound for LC polynomials} to obtain a lower bound for $B_{(\bm 0 , \bm k)}(M , \bm b) = \# \{ \bm y \in \Z^n: M \bm y = \bm b, \bm 0 \le \bm y \le \bm k \} $. Recall that a TU matrix $M$ is said to have finite fibers if $\{ \bm y \in \Z_{\ge 0}^n : M \bm y = \bm b\}$ is finite for any $\bm b$.

\begin{theorem}
    \label{initial lower bound for TU matrices using a lin transformation}
    Suppose $M \in \Z^{m \times n}$ is a full row rank TU matrix with finite fibers. Then for any $\bm b \in \Z^m$ and $\bm k \ge \bm 1$:
    \[
    B_{(\bm 0 , \bm k)}(M , \bm b) \ge \cpc_{\bm b}(P_{M, \bm k}) \cdot \left( \frac{\|\bm b \|_1^{\|\bm b \|_1}}{(1 + \|\bm b \|_1)^{1 +\|\bm b \|_1 }}\right)^m.
    \]
\end{theorem}
\begin{proof}
    Let $U \in \Z^{m \times m}$ be the matrix defined in \Cref{a good linear transformation for supp PM}, let $N := U M$ and $\bm \beta = U \bm b$. $N$ is a TU matrix by construction and $P_{N, \bm k}$ is $\LC$ by \Cref{k-solution Laurent polynomial of tu matrices are lc}. Using \Cref{capacity bound for LC polynomials} and  the second property of \Cref{a good linear transformation for supp PM} we get: 
    \[
    [\bm x^{\bm \beta}]P_{N, \bm k} \ge \cpc_{\bm \beta} P_{N, \bm k} \cdot \left( \frac{\|\bm b \|_1^{\|\bm b \|_1}}{(1 + \|\bm b \|_1)^{1 +\|\bm b \|_1 }}\right)^{m-1} \cdot \frac{|\beta_m|^{|\beta_m|}}{ (1 + |\beta_m|)^{1 + |\beta_m|}}. 
    \]
    Note that $t \mapsto \frac{t^t}{(1 + t)^{1 + t}}$ is a decreasing function over $\R_{\ge 0}$ and $\beta_m = (U \bm b)_m \le \|b \|_1$, so $\frac{|\beta_m|^{|\beta_m|}}{(1 + |\beta_m|)^{1 + |\beta_m|}} \ge \frac{\|\bm b \|_1^{\|\bm b \|_1}}{(1 + \|\bm b \|_1)^{1 +\|\bm b \|_1 }} $. The rest of the proof follows from the third and fourth property of \Cref{a good linear transformation for supp PM}.
\end{proof}

Our goal is to find a lower bound for $B(M, \bm b) = \{ \bm y \in \Z_{\ge 0}^n : M \bm y = \bm b\}$. To find such a lower bound, we will let $\bm k$ tend to $\infty$ in the bound presented in \Cref{initial lower bound for TU matrices using a lin transformation}. It is easy to see that the left hand side converges to $B(M, \bm b)$, that is, $\lim_{\bm k \to \infty}B_{(\bm 0 , \bm k)}(M , \bm b) = B(M , \bm b)$. We show that $\lim_{\bm k \to \infty} P_{M, \bm k}$ converges and find this limit. Let us recall some terminology related to Laurent series and their capacity function. 

Similar to the terminology of \cite{LMY26}, for a Laurent series $f(\bm x)$ with non-negative coefficients, we define its domain of convergence as $\Omega_f = \{\bm x>\bm 0: f(\bm x) \text{ converges } \}$, and define its capacity as $\cpc_{\bm {\alpha}} (f) = \inf_{\bm x \in \Omega_f}\frac{f(\bm x)}{\bm x^{\bm \alpha}}$. If we define $\bm f(\bm x) = \infty$ for any $\bm x \not \in \Omega_f$, this definition of capacity agrees with the earlier definition $\cpc_{\bm \alpha}(f) = \inf_{\bm x > \bm 0} \frac{f(\bm x)}{\bm x^{\bm \alpha}}$.

Let $\bm a_1, \cdots , \bm a_n$ be the columns of $M$ and define:
\[
P_M(\bm y) : =\sum_{\bm{y} \ge \bm{0}} \bm{x}^{M \bm{y}}  = \prod_{j = 1}^n \sum_{\ell \ge 0}\bm x^{\ell \bm a_j}.
\]
Suppose $M$ is an $m\times n$ TU matrix and has finite fibers, i.e.  $\{\bm y\in \Z_{\ge 0}^n: M \bm y = \bm b \}$ is finite for any $\bm b \in \Z^m$. Then $P_M$ is a well defined Laurent series with the domain of convergence $\Omega = \{ \bm{x} \in \R^m_{>0}: \bm x^{\bm a_j} < 1 \quad \forall j \in [n]\}$. Moreover, $\Omega$ is a nonempty set. By Motzkin's transposition theorem \cite[Theorem 7.17]{Gul10}, there exists some $\bm w$ such that $\bm w^T \bm a_j >0 $ for any $j \in [n]$ since $M$ has finite fibers. It follows that $e^{-r\bm w} \in \Omega$ for any $r > 0$. More precisely, the domain of convergence of each $\sum_{\ell \ge 0} \bm x^{\ell \bm a_j}$ is open and contains $e^{-r \bm w}$. 

\begin{lemma}[{\cite[Proposition 6.4]{LMY26}
}]
\label{capacity of products}
    Suppose $f_1(\bm x), \cdots, f_n(\bm x)$ are all Laurent series with non-negative coefficients. Let $\Omega_i \neq \emptyset$ be the domain of convergence of $f_i$, and suppose $\cap_{i = 1}^n \innt\Omega_i \neq \emptyset$. Then: 
    \[
    \cpc_{\bm {\alpha}}\left( \prod_{i =1 }^n f_i\right) = \sup_{\bm \alpha_{(1)} + \cdots + \alpha_{(n)}= \bm \alpha} \prod_{i = 1}^n \cpc_{\bm \alpha_{(i)}}(f_i) .\]
    In fact, after taking the natural logarithm of both hand sides, the two programs are Fenchel primal-dual. 
\end{lemma}

In our case, this implies a particularly nice expression for the capacity of $P_M$. We can also obtain a conceptually similar expression for the capacity of $P_{M,\bm{k}}$, which we discuss below.

\begin{corollary} \label{cap-sup-polytope}

    Fix an $m \times n$ TU matrix $M$ with finite fibers and an integer vector $\bm{b} \in \Z^m$. Letting $\mathcal{B}{(M,\bm{b})} = \{\bm{y} \in \R_{\ge 0}^n : M\bm{y} = \bm{b}\}$, we have 
    \[
        \cpc_{\bm{b}}(P_M) = \sup_{\bm{s} \in \mathcal{B}{(M,\bm{b})}} \prod_{j=1}^n \frac{(s_j+1)^{s_j+1}}{s_j^{s_j}}.
    \]
\end{corollary}
\begin{proof}
    Let $\bm a_1, \cdots, \bm a_n \in \Z^m$ be the columns of $M$. By \Cref{capacity of products} we can write: 
    \[
    \cpc_{\bm b}(P_{M}) = \sup_{\bm b_1 + \cdots + \bm b_n = \bm b} \prod_{j = 1}^n \cpc_{\bm b_j} \left( \sum_{\ell \ge 0} \bm x^{\ell \bm a_j}\right).
    \]
    We can further restrict to the case where $\bm b_j$ is a non-negative multiple of $\bm a_j$ by \Cref{capacity is nonzero iff alpha is in newton polytope}, otherwise, $\cpc_{\bm b_j}\left( \sum_{\ell \ge 0}\bm x^{\ell \bm a_j} \right) = 0$. It follows that: 
    \[
    \cpc_{\bm b}(P_{M}) = \sup_{\bm s \in \mathcal{B}(M, \bm b)} \prod_{j = 1}^n \cpc_{s_j \bm a_j} \left( \sum_{\ell \ge 0} \bm x^{\ell \bm a_j}\right).
    \]
    One can see that $\cpc_{s_j \bm a_j} \left( \sum_{\ell \ge 0} \bm x^{\ell \bm a_j}\right) = \cpc_{\bm s_j} \left(\sum_{\ell \ge 0} x^{\ell}\right)$ by a change of variables. We further have $\cpc_{s_j}\left(\sum_{\ell \ge 0} x^{\ell}\right) = \frac{(s_j + 1)^{s_j + 1}}{s_j^{s_j}}$ by \cite[Example 4.16]{LMY26}, completing the proof. 
\end{proof}

We note that there is a similar expression for \Cref{cap-sup-polytope} in the case of the polytope $\mathcal{B}_{\bm{0}, \bm k}{(M,\bm{b})} = \{\bm{y} \in \R^n : M\bm{y} = \bm{b}, \, \bm{0} \leq \bm{y} \leq \bm{k}\}$. In that case, the capacity of $P_{M,\bm{k}}$ is the $\sup$ over $\bm{s} \in \mathcal{B}_{\bm{0}, \bm k}{(M,\bm{b})}$ of the product of $e^{\mathcal{H}_{\{0,\ldots,k_j\}}(s_j)}$, where $\mathcal{H}_{\{0,\ldots,k_j\}}(s_j)$ is the entropy of the entropy-maximizing distribution on $\{0,\ldots,k_j\}$ with expectation $s_j$ (Note that this is precisely what is given in \Cref{cap-sup-polytope}, for the geometric distribution on $\Z_{\geq 0}$ with expectation $s_j$.)

Now we can use \Cref{capacity of products} to prove that capacity is continuous under limiting $\bm{k} \to \infty$.

\begin{lemma}
    \label{cap of PMk converges to cap PM}
    Let $M$ be an $m \times n$ TU matrix with finite fibers and fix $\bm b \in  \Z^m$. Then:
    \[
    \lim_{\bm k \to \infty} \cpc_{\bm{b}}(P_{M, \bm k}) = \cpc_{\bm b} P_M.
    \]
    
\end{lemma}
\begin{proof}
Observe that $P_M(\bm x) \ge P_{M, \bm k_1}(\bm x) \ge P_{M, \bm {k}_2}(\bm x)$ for any $\bm x>\bm 0$ and $\bm k_1 \ge \bm k_2$ (for $\bm x \not \in \Omega$, let $P_M(\bm x) = \infty$), so $\cpc_{\bm b} (P_{M}) \ge \cpc_{\bm b}(P_{M,  \bm k}) \ge 0$ for any $\bm k$. Since $\cpc_{\bm b}(P_{M,\bm k})$ is monotone in $\bm k$, the limit $\lim_{\bm k \to \infty}\cpc_{\bm b}(P_{M,\bm k})$ exists and $\cpc_{\bm b}(P_M) \ge \lim_{\bm k \to \infty}\cpc_{\bm b}(P_{M,\bm k})$.

For the other direction of the proof, let $\bm a_1, \cdots, \bm a_n \in \Z^m$ be the columns of $M$. 
By \Cref{capacity of products} and \Cref{capacity is nonzero iff alpha is in newton polytope}:
\begin{align*}
\cpc_{\bm b}(P_{M, \bm k}) &= \sup_{\bm s \in \mathcal{B}(M, \bm b)} \prod_{j = 1}^n \cpc_{s_j\bm a_j} \left( \sum_{\ell = 0}^{k_j} \bm x^{\ell \bm a_j}\right),\\
\cpc_{\bm b}(P_{M}) &= \sup_{\bm s \in \mathcal{B}(M, \bm b)} \prod_{j = 1}^n \cpc_{s_j\bm a_j} \left( \sum_{\ell \ge 0} \bm x^{\ell \bm a_j}\right).
\end{align*}
Note that in the expression for $\cpc_{\bm b}(P_{M,\bm k})$, we can take the supremum to be over $\mathcal{B}(M, \bm b)$ instead of $\mathcal{B}_{(\bm 0 , \bm k)}(M, \bm b)$ because for any $\bm s \in \mathcal{B}(M, \bm b) \setminus \mathcal{B}_{\bm 0 , \bm k}(M, \bm b)$, the corresponding product of capacities on the right hand side is zero. 

Similar to the proof of \Cref{cap-sup-polytope}, write $\cpc_{s_j\bm a_j}( \sum_{\ell = 0}^{k_j} \bm x^{\ell \bm a_j}) = \cpc_{s_j} (\sum_{\ell = 0}^{k_j} x^{\ell})$. Therefore, for any choice of $\bm s \ge \bm 0$ such that $M \bm s = \bm b$ we have: 
\[
\cpc_{\bm b} (P_{M, \bm k}) \ge \prod_{j = 1}^n \cpc_{s_j}\left( \sum_{\ell = 0}^{k_j} x^{\ell}\right),
\]

 By \cite[Lemma 5.5]{BLP23}, $\lim_{k_j \to \infty} \cpc_{s_j} (\sum_{\ell = 0}^{k_j} x^{\ell}) = \cpc_{s_j}\left( \sum_{\ell \ge 0} x^{\ell}\right)$. Send $k_j \to \infty$ for all $j \in [n]$ and write: 
\[ \lim_{\bm k \to \infty} \cpc_{\bm b}(P_{M, \bm k}) \ge \prod_{j = 1}^n \cpc_{s_j}\left( \sum_{\ell \ge 0} x^{\ell}\right),\]
Note that this is true for any choice of $\bm s \in \mathcal{B}(M, \bm b)$, so:
\begin{align*}
    \lim_{\bm k \to \infty} \cpc_{\bm b}(P_{M, \bm k}) &\ge \sup_{\bm s \in \mathcal{B}(M, \bm b)} \prod_{j = 1}^n \cpc_{s_j}\left( \sum_{\ell \ge 0} x^{\ell}\right) \\
    &= \cpc_{\bm b}(P_M).
\end{align*}

\end{proof}
We conclude the following corollary using \Cref{initial lower bound for TU matrices using a lin transformation} and \Cref{cap of PMk converges to cap PM}: 
\begin{corollary}
\label{secondary lower bound using a linear transformation}
    Suppose $M \in \Z^{m \times n}$ is a full row rank matrix with finite fibers. For any $\bm b \in \Z^m$:
    \[
    B(M , \bm b) \ge \cpc_{\bm b}(P_M) \cdot \left( \frac{\|\bm b \|_1^{\|\bm b \|_1}}{(1 +\|\bm b \|_1)^{1 +\|\bm b \|_1  } }\right)^m.
    \]
\end{corollary}

We can further use the bound of \Cref{secondary lower bound using a linear transformation} $\frac{k^k}{(1 + k)^{1 + k}} \ge \frac{1}{e(k+1)}$ for $k \ge 0$, and to obtain a simply exponential lower bound for $B(M, \bm b)$ when $\| \bm b\|_1$ is large enough:

\begin{corollary}
    \label{simply exponential bound using a linear transformation}
    Suppose $M \in \Z^{m \times n}$ is a full row rank matrix with finite fibers. For any $\bm b \in \Z^m$:
    \[
    B(M, \bm b) \ge \cpc_{\bm b}(P_M) \cdot \exp( - m - m \log \| \bm b \|_1 ).
    \] 
    If $m \ge 2$ and $\| \bm b\|_1 \ge 2m \log(2m)$, then:
    \[
    B(M, \bm b) \ge \cpc_{\bm b}(P_M) \cdot \exp(- \| \bm b\|_1). 
    \]
\end{corollary}
\begin{proof}
    The first statement follows directly from \Cref{secondary lower bound using a linear transformation} and the inequality $\frac{k^k}{(1 + k)^{1 + k}} \ge \frac{1}{e(k+1)}$. For the second inequality, note that if $m \ge 2$ and $\| \bm b\| _1 \ge 2m \log(2m)$. For brevity, let $c = \|b \|_1$. It suffices to show that $\exp( - m - m  \log c) \ge \exp (- c) $, or equivalently, that $c \ge m + m\log c$. Note that $t \mapsto t - m - m \log t$ is an increasing function over $t \ge m$. Hence we just need to show that $c \ge m + m \log c$ when $c = 2m \log 2m$. Write:
    \begin{align*}
        c - m - m \log c = m \left (2 \log (2m) - 1 - \log(2m) - \log \log (2m) \right) \ge 0
    \end{align*}
    where we are using the fact that $t - 1 - \log(t) \ge 0$ for $t \ge 1$.
\end{proof}

\subsection{Tightness of the bound} \label{sec:tightness}

We show that the capacity lower bound of \Cref{main-thm} has tight dependence on $\|\bm{b}\|_1$, via a specific example. Let $B$ be the $m \times m$ matrix which has $1$s on the main diagonal, $-1$s on the subdiagonal, and $0$s elsewhere. Further, let $M$ be the $m \times 2m$ matrix which given by $M = \begin{bmatrix} B & B \end{bmatrix}$, and let $\bm{b} = k\bm{e}_1$ for some $k$. Concretely,
\[
    M = \begin{bmatrix}
        1 & 0 & 0 & \cdots & 0 & 1 & 0 & 0 & \cdots & 0 \\
        -1 & 1 & 0 & \cdots & 0 & -1 & 1 & 0 & \cdots & 0 \\
        0 & -1 & 1 & \cdots & 0 & 0 & -1 & 1 & \cdots & 0 \\
        \vdots & \vdots & \vdots & \ddots & \vdots & \vdots & \vdots & \vdots & \ddots & \vdots \\
        0 & 0 & 0 & \cdots & 1 & 0 & 0 & 0 & \cdots & 1
    \end{bmatrix},
    \quad
    \bm{b} = \begin{bmatrix}
        k \\ 0 \\ 0 \\ \vdots \\ 0
    \end{bmatrix}.
\]
Note that $B$ is a directed acyclic graph incidence matrix with a row removed, and thus $M$ is TU. Also $M$ has full row-rank.

Note that a non-negative integral solution $\bm{x}$ to $M\bm{x} = \bm{b}$ is equivalent to a choice of $m$ pairs $(x_1,x_{m+1}),(x_2,x_{m+2}),\ldots$ of non-negative integers adding to $k$. Thus
\[
    \#\{\bm{x} \in \Z_{\geq 0}^{2m} : M\bm{x} = \bm{b}\} = (k+1)^m.
\]
Next we want to compute
\[
    \inf_{0 < x_1 < \cdots < x_m < 1} \left(\frac{1}{x_1^k} \left(\frac{1}{1-x_m} \prod_{j=1}^{m-1} \frac{1}{1-\frac{x_j}{x_{j+1}}}\right)^2\right) = \inf_{0 < y_1,\ldots,y_m < 1} \left(\prod_{j=1}^m \frac{1}{y_j^k(1-y_j)^2}\right),
\]
where the equality comes from setting $y_j = \frac{x_j}{x_{j+1}}$ with $x_{m+1} = 1$. This optimization problem is then separable variable-wise, and thus we have
\[
    \inf_{0 < y_1,\ldots,y_m < 1} \left(\prod_{j=1}^m \frac{1}{y_j^k(1-y_j)^2}\right) = \left[\inf_{0 < t < 1} \frac{1}{t^k(1-t)^2}\right]^m = \left(\frac{(k+2)^{k+2}}{4k^k}\right)^m.
\]
Thus \Cref{main-thm} gives a lower bound of
\[
    (k+1)^m = \#\{\bm{x} \in \Z_{\geq 0}^{2m} : M\bm{x} = \bm{b}\} \geq \left(\frac{(k+2)^{k+2}}{4e(k+1)k^k}\right)^m.
\]
This bound is then tight up to simply exponential factor independent of $k = \|\bm{b}\|_1$, which implies the claimed tightness of the bound.

\section{Efficiency of Approximation}
\label{efficiency-proof}

We prove the algorithmic assertion of \Cref{main-algo} using the
interiority and ellipsoid argument of \cite{SV14}
(see Sections~5,6 of the arXiv version). We verify the
necessary bounds directly for our Laurent polynomials, rather
than applying their statements for $0/1$ configurations verbatim.
Throughout this section we assume $M$ is full row-rank, but this is without loss of generality since one can always remove linearly dependent rows of $M$ without changing the lattice point counts or the value of the capacity (up to consistency of $\bm{b}$).

Let $M\in\mathbb Z^{m\times n}$ be TU with columns $\bm{a}_1,\ldots,\bm{a}_n$, let $\bm b \in \Z^m$ and
$\bm k \in \Z_{\geq 0}^n$, and write
\[
    p(\bm x)=P_{M,\bm{k}}(\bm x)
    =\prod_{j=1}^n\sum_{\ell=0}^{k_j}\bm{x}^{\ell \bm{a}_j},
    \qquad
    \mathcal{E}=\cpc_{\bm{b}}(p).
\]
We show that a deterministic algorithm can approximate $\log \mathcal{E}$ to additive error
$2^{-q}$ in time polynomial in
\[
    m,\quad n,\quad q = \log \frac{1}{\epsilon},\quad
    \log(1+\|\bm{b}\|_\infty),\quad
    \log(1+\|\bm{k}\|_\infty),
\]
as long as $\mathcal{E} \neq 0$.
All inputs are binary encoded and running time is measured in
bit operations.

Finally, for \Cref{main-algo}, set
$K=\|\bm{b}\|_1$ and $k_j=K$ before
performing any reductions. Note that the $\bm{x}^{\bm{b}}$ coefficient of $P_{M,\bm{k}}(\bm{x})$ is equal to $\#\{\bm{y} \in \Z_{\geq 0}^n \mid M\bm{y} = \bm{b}\}$ for this value of $\bm{k}$.) This gives the stated dependence on
$\log(1+\|\bm{b}\|_1)$.

\paragraph{Reduction to $\bm{b}$ in the interior.}
The Newton polytope of $p$ is the zonotope $Z=M[\bm{0},\bm{k}]$; i.e. the sum of the line segments defined by $k_j \bm{a}_j$ for $j \in [n]$.

Let $F$ be the smallest face of $Z$ containing $\bm{b}$. Minimizing and maximizing each coordinate over $Q = \{\bm{y}\in\R^n:M\bm{y}=\bm{b},\ \bm{0}\leq \bm{y}\leq \bm{k}\}$ identifies the set of coordinates forced to equal either $0$ and $k_j$. Thus we can remove these columns of $M,\bm{k}$ to obtain $M',\bm{k}'$, and up to translation the zonotope $M'[\bm{0},\bm{k}']$ is equal to $F$. Restricting to $F$ and the associated variables preserves the corresponding coefficient and the value of the capacity by a standard optimization argument. Additionally, $\bm{b}$ is in the relative interior of the Newton polytope of the new polynomial, up to the same translation between $M'[\bm{0},\bm{k}']$ and $F$. Thus we may WLOG assume that $\bm{b}$ is in the interior of $Z$.

Finally, note that this removal of columns may decrease the rank of $M$. Thus, we again remove linearly dependent rows of $M$ at this point to account for this.

\paragraph{Bounding box for the minimizer.}
By \Cref{zonotope-normals} (see below), each facet inequality of $Z$
can be written with normal $\bm{u}$ as
\[
    \bm{u}^\top\bm{y}\leq h,
    \quad
    \bm{u}\in\{0,+1,-1\}^m, \quad h\in\mathbb Z.
\]
Since $\bm{b}$ is in the interior of $Z$, its slack in each inequality is at
least one. Consequently,
\[
    B_2\left(\bm{b},\frac{1}{\sqrt{m}}\right)\subseteq Z.
\]
Now set $\Phi(\bm{z})=\log p(e^{\bm{z}})-\bm{b}^\top \bm{z}$, noting that $\log \cpc_{\bm{b}}(p) = \inf_{\bm{z} \in \R^m} \Phi(\bm{z})$.
Every nonzero coefficient of $p$ is a positive integer, and so
\[
    \Phi(\bm{z})
    \geq \max_{\bm\alpha\in Z}(\bm\alpha-\bm{b})^\top \bm{z}
    \geq \frac{\|\bm{z}\|_2}{\sqrt{m}}.
\]
Thus $\Phi$ attains its minimum on $\R^m$, and every minimizer $\bm{z}^*$ satisfies
\[
    \|\bm{z}^*\|_2\leq \sqrt{m} \cdot \log p(\bm{1}) = \sqrt{m} \sum_{j=1}^n \log(k_j+1).
\]
In particular, we may optimize over $[-R,R]^m$, where
\[
    R = \sqrt{m} \sum_{j=1}^n\log(k_j+1).
\]
This is the interiority-based bounding argument of
\cite{SV14}, with $\log p(\mathbf1)$ replacing the logarithm
of the number of configurations.

\begin{lemma} \label{zonotope-normals}
    Let $M$ be an $m \times n$ totally unimodular matrix of rank $m$ with columns $\bm{a}_1,\ldots,\bm{a}_n$, fix $\bm{k} \in \Z_{>0}^n$, and define the zonotope $Z = M[\bm{0},\bm{k}] = \sum_{j=1}^n \bm{a}_j [0,k_j]$. Then the normal vectors of the facets of $Z$ can be chosen to have entries in $\{0,+1,-1\}$.
\end{lemma}
\begin{proof}
    Let $F$ be some facet of $Z$ with normal vector $\bm{u}$. By \cite{McMullen71}, every facet of the zonotope $Z$ is parallel to a Minkowski sum of some subset of the segments $\{\bm{a}_1[0,k_1],\ldots,\bm{a}_n[0,k_n]\}$. Thus $F$ is parallel to $\mathrm{span}\{\bm{a}_j : \bm{u}^\top \bm{a}_j = 0\}$. WLOG by reindexing let $\bm{a}_1,\ldots,\bm{a}_{m-1}$ span the $(m-1)$-dimensional space parallel to $F$, and let $C$ be the $m \times (m-1)$ submatrix of the corresponding columns of $M$, so that $\bm{u} \in \ker C^\top$. Define $v_i = (-1)^{i+1} \det C_{\hat{i}}$ for all $i \in [m]$, where $C_{\hat{i}}$ is the matrix $C$ with row $i$ missing. A straightforward computation shows that $\bm{v}^\top \bm{a}_j = 0$ for all $j \in [m-1]$ (duplicate a given column of $C$ and take determinant), and thus $\pm \bm{v}$ can be chosen as the normal vector of $F$. Since $M$, and thus $C$, is TU, this implies the desired result.
\end{proof}

\paragraph{Approximate evaluation.}

We show that, for rational $\bm{z}$ and $0<\tau<1$, the objective
\[
    \Phi(\bm{z})=-\bm{b}^\top \bm{z} + \sum_{j=1}^n H_{k_j+1}(\bm{a}_j^\top \bm{z}),
    \qquad
    H_N(t)=\log\sum_{r=0}^{N-1}e^{rt},
\]
can be evaluated to additive error $\tau$ in time polynomial
in the input length, the encoding length of $\bm{z}$, and
$\log\frac{1}{\tau}$. The dot products are computed exactly using
rational arithmetic. To evaluate $H_N$, use $H_N(0)=\log N$
and, for $t\neq0$, the geometric-series identity
\[
    H_N(t)
    =(N-1)\max\{t,0\}
      +\ell(N|t|)-\ell(|t|),
    \qquad
    \ell(s)=\log(1-e^{-s}).
\]
This avoids forming exponentials with large positive arguments.

We describe the precision needed to evaluate either
$\ell(s)$, where $s\in\{|t|,N|t|\}$, to error $2^{-q}$.
Let $B$ be the binary encoding length of the nonzero
rational number $t$. Since $s\ge |t|\ge2^{-B}$, we have
\[
    u:=1-e^{-s}
    \ge \tfrac12\min\{s,1\}
    \ge 2^{-B-1}.
\]
If $s\ge q+3$, then
\[
    |\ell(s)|\le 2e^{-s}\le2^{-q-2},
\]
so returning zero suffices. Otherwise $0<s<q+3$.
Approximate $e^{-s}$ to absolute error
$\eta=2^{-(B+q+4)}$ via Taylor expansion; note that $s$ is of magnitude $O(q)$ and
requires $O(B+q)$ fractional bits. Now subtract this approximation
from $1$ to obtain $\widetilde u$, which implies
$|\widetilde u-u|\le\eta\le \frac{u}{2}$ and thus $\widetilde u \geq \frac{u}{2}$. Therefore by the mean value theorem, for some $c$ between $u$ and $\widetilde u$ we have
\[
    |\log\widetilde u-\log u|
    \leq |\widetilde u - u| \cdot (\log t)'(c)
    \le \eta \cdot \frac{2}{u}
    \le 2^{-q-2}.
\]
Approximate $\log\widetilde u$ to an additional error of
$2^{-q-2}$ via Taylor expansion of $\log(1-t)$, by first scaling the argument $\widetilde u$ to the interval $[\frac{1}{2},1]$. This gives the required approximation to $\ell(s)$.

The rational term $(N-1)\max\{t,0\}$ has encoding length
polynomial in $B+\log N$, and errors in the two logarithms
add without amplification. Thus $H_N(t)$ is computable to
error $2^{1-q}$ in time polynomial in $B,\log N,q$.
Taking
\[
    q=\left\lceil\log_2\frac{4n}{\tau}\right\rceil
\]
bounds the error in each term of the sum defining $\Phi(\bm{z})$ by $\frac{\tau}{2n}$;
the cases $t=0$ are evaluated to the same or better accuracy.
Exact rational summation produces a rational $V$ satisfying
$|V-\Phi(z)|\le\frac{\tau}{2} < \tau$.




\paragraph{Approximate gradients.}

Approximate gradients then follow from finite differences, which we now show. First, since $\bm{b} \in M[\bm{0},\bm{k}]$, let $\bm{y} \in [\bm{0},\bm{k}]$ be such that $\sum_{i=1}^n y_i \bm{a}_i = \bm{b}$.
For fixed $k$ and $t$, the first and second derivatives of
$H_{k+1}$ are the mean and variance of the distribution $\mu$
supported on $\{0,\ldots,k\}$ with probabilities $\mu_\ell \propto e^{\ell t}$. Since $\mu$ is supported on $\{0,\ldots,k\}$, this implies
\[
    \|\nabla \Phi(\bm{z})\|_2 = \left\|\sum_{j=1}^n \bm{a}_j \left[H_{k_j+1}'(\bm{a}_j^\top \bm{z}) - y_j\right]\right\|_2 \leq L := \sqrt{m} \sum_{j=1}^n k_j.
\]
and
\[
    0 \preceq \nabla^2 \Phi(\bm{z}) = \sum_{j=1}^n \bm{a}_j \bm{a}_j^\top H_{k_j+1}''(\bm{a}_j^\top \bm{z}) \preceq I \cdot m \sum_{j=1}^n k_j^2,
\]
so that
\[
    \|\nabla^2 \Phi(\bm{z})\|_{\mathrm{op}} \leq H := m \sum_{j=1}^n k_j^2.
\]
For coordinatewise gradient error $\delta$, evaluate $\Phi(\bm{z})$ and
$\Phi(\bm{z} + h\bm{e}_i)$, so that
\[
    \left|\partial_{z_i} \Phi(\bm{z}) - \frac{\Phi(\bm{z} + h\bm{e}_i) - \Phi(\bm{z})}{h}\right| \leq \frac{H h}{2}.
\]
Approximating $\Phi(\bm{z})$ by $V_0$ and $\Phi(\bm{z} + h\bm{e}_i)$ by $V_i$ within error $\tau$ (as described above) then implies $\frac{V_i-V_0}{h}$ is within error $\frac{2\tau}{h}$. This implies
\[
    \left|\partial_{z_i} \Phi(\bm{z}) - \frac{V_i - V_0}{h}\right| \leq \frac{H h}{2} + \frac{2\tau}{h}
\]
Taking $h = \frac{\delta}{4H}$ and $\tau = \frac{\delta h}{8} = \frac{\delta^2}{32H}$, this implies an overall error of $\frac{H h}{2} + \frac{2\tau}{h} = \frac{\delta}{8} + \frac{\delta}{4} < \delta$. Note that $H$ is polynomial in $m$ and $\bm{k}$, and thus $\log \frac{1}{\tau}$ depends polynomially on $\log m$, $\log(1+\|\bm{k}\|_1)$, and $\log \frac{1}{\delta}$.

\paragraph{The algorithm.}

We now apply the approximate-oracle ellipsoid implementation
of Section~6.2 (of the arXiv version) of \cite{SV14} (see also Theorem~2.13), using the
preceding bounds on the search region and the Lipschitz
constant $L$ of $\Phi$. The value and gradient oracles established
above provide its required numerical approximations.
The number of oracle calls and the required number of accuracy
bits are polynomial in the input length and
$\log \frac{1}{\epsilon}$. Since our oracles run in time polynomial
in these quantities and the encoding length of the query point,
the resulting algorithm finds $\bm{z}^\circ$ satisfying
\[
    \Phi(\bm{z}^\circ)\leq\min\Phi+\frac{\epsilon}{2}
\]
in polynomial bit complexity.
Evaluating $\Phi(\bm{z}^\circ)$ to additive error $\frac{\epsilon}{2}$
then gives a rational $\mathcal{A}$ with
$|\mathcal{A}-\log \mathcal{E}|\leq\epsilon$.

We note that although the total running time stated in \cite{SV14}
depends polynomially on $\frac{1}{\epsilon}$, its ellipsoid
iteration bound depends only polynomially on
$\log \frac{1}{\varepsilon}$ (see Section~6.2, immediately before
Lemma~6.3). The worse dependence comes from the
approximate-counting oracle used there. Our value and
gradient oracles instead run in time polynomial in the
number of requested accuracy bits, so the same optimization
framework gives polynomial dependence on
$\log \frac{1}{\varepsilon}$ in our setting.
\subsection*{Acknowledgments}

Both authors acknowledge the support of the Natural Sciences and Engineering Research Council of Canada (NSERC), [funding reference number RGPIN-2023-03726]. Cette recherche a \'et\'e partiellement financ\'ee par le Conseil de recherches en sciences naturelles et en g\'enie du Canada (CRSNG), [num\'ero de r\'ef\'erence RGPIN-2023-03726].

\bibliographystyle{amsalpha}
\bibliography{Ref}

\end{document}